\documentclass[11pt,a4paper]{amsart}
\usepackage[left=35mm,right=35mm]{geometry}
\usepackage{amsmath,amssymb,amsthm,stmaryrd}
\usepackage[all]{xy}
\usepackage{enumerate}
\usepackage{hyperref}
\newcommand{\kk}{\mathbf{k}}
\newcommand{\Aut}{\operatorname{Aut}}
\newcommand{\calP}{\mathcal{P}}
\newcommand{\calQ}{\mathcal{Q}}
\newcommand{\Ker}{\operatorname{Ker}}
\newcommand{\Ima}{\operatorname{Im}}
\newcommand{\Tot}{\operatorname{Tot}}

\newcommand{\Ext}{\operatorname{Ext}}
\newcommand{\Hom}{\operatorname{Hom}}
\newcommand{\id}{\operatorname{id}}
\newcommand{\rmd}{\mathrm{d}}
\newcommand{\llb}{\llbracket}
\newcommand{\rrb}{\rrbracket}
\newcommand{\varphantom}[1]{\mathrel{\phantom{#1}}}
\newtheorem{prop}{Proposition}[section]
\newtheorem{lem}[prop]{Lemma}
\newtheorem{thm}[prop]{Theorem}
\newtheorem{cor}[prop]{Corollary}
\theoremstyle{definition}
\newtheorem{defn}[prop]{Definition}
\theoremstyle{remark}
\newtheorem{remark}[prop]{Remark}
\numberwithin{equation}{section}
\allowdisplaybreaks
\begin{document}

\setlength{\baselineskip}{1.4em}

\title[Homological smoothness of generalized Weyl algebras]{Homological smoothness and deformations\\ of generalized Weyl algebras}
\author{L.-Y.~Liu}
\address{Departement Wiskunde-Informatica, Middelheimcampus, Universiteit Antwerpen, Middelheimlaan 1, 2020 Antwerp, Belgium}
\email{liyu.liu@uantwerpen.be}
\thanks{The author acknowledges the support of the European Union for ERC grant No 257004-HHNcdMir.}
\begin{abstract}
It is an immediate conclusion from Bavula's papers \cite{Bavula:GWA-def}, \cite{Bavula:GWA-tensor-product} that if a generalized Weyl algebra $A=\kk[z;\lambda,\eta,\varphi(z)]$ is homologically smooth, then the polynomial $\varphi(z)$ has no multiple roots. We prove in this paper that the converse is also true. Moreover, formal deformations of $A$ are studied when $\kk$ is of characteristic zero.
\end{abstract}
% 16E10 Homological dimension
% 16E40 (Co)homology of rings and algebras (e.g. Hochschild, cyclic, dihedral, etc.)
% 16S80 Deformations of rings
\maketitle

\section{Introduction}
During the development of algebra, an impetus is to introduce and study noncommutative objects with commutative background. Among these noncommutative objects, a class of algebras---generalized Weyl algebras, introduced by Bavula in \cite{Bavula:GWA-def}---have been studied from different points of view. There are many examples of generalized Weyl algebras related to rings of differential operators or quantum groups, such as the usual Weyl algebras, quantum planes, and quantum spheres.

Roughly speaking, let $B$ be an algebra, $\sigma$ an algebra automorphism of $B$ and $a$ a central element in $B$. The triple $(B,\sigma,a)$ determines a generalized Weyl algebra, which is generated by two variables $x$, $y$ over $B$ subject to some relations. Starting with the same $B$, generalized Weyl algebras may have different ring-theoretic and/or homological properties when $\sigma$, $a$ vary. The case $B=\kk[z]$ has received much attention where $\kk$ is a field. Necessarily, $\sigma(z)=\lambda z+\eta$ for some $\lambda\in\kk\setminus\{0\}$ and $\eta\in\kk$, $a$ is a polynomial $\varphi(z)$. The resulting generalized Weyl algebras are denoted by $\kk[z;\lambda,\eta,\varphi(z)]$ in this paper. Here we only mention some references on their homological properties. It is illustrated in \cite{Bavula:GWA-tensor-product} that their global dimensions are equal to $1$, $2$ or $\infty$, and the latter occurs if and only if $\varphi(z)$ admits a multiple root (see also \cite{Bavula:GWA-def}, \cite{Hodge:Kleinian-singularities}, \cite{Solotar:Hochschild-homology-GWA-quantum}). Their Hochschild homology and cohomology are computed in \cite{Farinati:Hochschid-homology-GWA}, \cite{Solotar:Hochschild-homology-GWA-quantum}. In particular, a remarkable result in \cite{Farinati:Hochschid-homology-GWA} is that to assure a duality between its Hochschild homology and cohomology, the algebra $\kk[z;1,\eta,\varphi(z)]$ should have finite projective dimension as a bimodule over itself (also called Hochschild cohomological dimension). This dimension is related to Van den Bergh duality.

An algebra $A$ is said to be homologically smooth if $A$ has a finitely generated projective resolution of finite length as an $A^e$-module. Van den Bergh proved in \cite{Van-den-Bergh:VdB-duality} that over a homologically smooth algebra $A$, if there is an invertible bimodule $U$ and $d\geq 0$ such that $\mathrm{RHom}_{A^e}(A,A^e)\cong U[-d]$ in the derived category $\mathbf{D}^b(A^e)$, then the duality $H^i(A,M)\cong H_{d-i}(A,U\otimes_AM)$ holds for any $A$-bimodule $M$. Plenty of algebras enjoy Van den Bergh duality, such as noetherian Artin-Schelter regular connected graded algebras, noetherian Artin-Schelter regular Hopf algebras, and some filtered algebras. Clearly, the Hochschild cohomological dimension of an algebra is not less than its global dimension. So the condition that $\varphi(z)$ has no multiple roots is necessary if $A=\kk[z;\lambda,\eta,\varphi(z)]$ is homologically smooth. It is natural to ask whether it is also a sufficient condition. It turns out in \cite{Farinati:Hochschid-homology-GWA} and \cite{Solotar:Hochschild-homology-GWA-quantum} that under the condition, $\sup\{n\mid H^n(A,A)\neq 0\}$ is finite. However, the fact is not sufficient to conclude that $A$ is homologically smooth.

Some examples of generalized Weyl algebras, say Weyl algebra $A_1(\kk)$, quantum $2$-plane $\kk_q[x,y]$ and the localization $\kk_q[x,y^{\pm 1}]$, are all homologically smooth. In \cite{Krahmer:qh-space}, Kr\"ahmer proved the standard quantum $2$-sphere is Artin-Schelter Gorenstein and homologically smooth. He also asked whether the non-standard ones are also homologically smooth. The above known results are obtained according to the individual algebraic structures of these algebras. Their Hochschild cohomological dimensions are all equal to $2$. So together with the results of \cite{Bavula:GWA-tensor-product}, \cite{Farinati:Hochschid-homology-GWA}, \cite{Hodge:Kleinian-singularities}, \cite{Solotar:Hochschild-homology-GWA-quantum}, we conjecture that the dimension equals $2$ provided that it is finite. Thus the problem boils down to the cohomology group $H^3(A,M)$ for an arbitrary $A$-bimodule $M$. In this paper, we make use of the periodic projective resolution constructed in Sect.\ \ref{sec:projective-resolution}, succeed in proving the sufficiency. Our tool is called homotopy double complex, which seems feasible to have other applications. Moreover, homologically smooth generalized Weyl algebras are proved to be twisted Calabi-Yau (Proposition \ref{prop:dualizing-complex} and Theorem \ref{thm:twisted-CY}).

\begin{thm}
Let $A=\kk[z;\lambda,\eta,\varphi(z)]$. The Hochschild cohomological dimension of $A$ is $2$ if $\varphi(z)$ has no multiple roots, and hence $A$ is homologically smooth. Moreover, $A$ is $\nu$-twisted Calabi-Yau where $\nu(x)=\lambda x$, $\nu(y)=\lambda^{-1}y$ and $\nu(z)=z$.
\end{thm}

In particular, quantum $2$-spheres are all homologically smooth. This gives a positive answer to \cite[Question 2]{Krahmer:qh-space}.

This paper is also dedicated to the deformations of $\kk[z;\lambda,\eta,\varphi(z)]$. Deformation theory is a system studying how an object in a certain category of spaces can be varied in dependence on the points of a parameter space. It deals with the structure of families of objects like varieties, singularities, vector bundles, presheaves, algebras or differentiable maps. Deformation problems appear in various areas of mathematics, in particular in algebra, algebraic and analytic geometry, and mathematical physics.

We care about formal deformations of associative algebras. Historically, the theme has its root in the work of Gerstenhaber \cite{Gerstenhaber:deformation}. There are closed connections between deformation theory and Hochschild cohomology. For example, the second Hochschild cohomology group may be interpreted as the set of (equivalence classes of) first order deformations; the obstruction theory is related to the third Hochschild cohomology group. In this paper, we construct formal deformations of noncommutative generalized Weyl algebras $A=\kk[z;\lambda,\eta,\varphi(z)]$ when $\kk$ is of characteristic zero. Generally speaking, a formal deformation can only be constructed under the severe condition that all obstructions are passed successfully. This condition trivially holds for homologically smooth generalized Weyl algebras (see \S\ref{subsec:deformation-survey} and Theorem \ref{thm:twisted-CY}). According to our computations, a surprising result is that even if $A$ is not homologically smooth, we can construct a formal deformation of $A$ starting with a special Hochschild $2$-cocycle. Concretely, we use a periodic complex to compute the Hochschild cohomology of $A$ with coefficients in any $A$-bimodule $M$ whose cocycles are called Per cocycles, and define a map $f\colon M\to\{\text{Per $2$-cocycles}\}$. When $M=A$, by a pair of quasi-isomorphisms between the Hochschild cochain complex and the periodic complex, we prove (see Theorems \ref{thm:formal-deformation-quantum} and \ref{thm:formal-deformation-classical})

\begin{thm}
  Let $A=\kk[z;\lambda,\eta,\varphi(z)]$ be a noncommutative generalized Weyl algebra. Let $F_1$ be the Hochschild $2$-cocycle corresponding to $f(z)$ if $\eta=0$, or to $f(1)$ if $\lambda=1$. There exist a family of $\kk$-bilinear maps $F_n\colon A\times A\to A$, $n\geq 2$ integrating $F_1$ that determine a formal deformation of $A$. The family $\{F_n\}$ is unique if each $F_n$ satisfies the conditions (a), (b) in Lemma \ref{lem:determine-F_n}.
\end{thm}

We also proved that in most cases, the Hochschild $2$-cocycle $F_1$ in the theorem is not a coboundary, and so our construction is not equivalent to the trivial one.

An interesting subject in noncommutative geometry and mathematical physics is to study how to treat a noncommutative object as a deformation of a commutative one. A class of generalized Weyl algebras ($\lambda=\eta=1$) were studied by T.J.~Hodges as noncommutative deformations of type-$A$ Kleinian singularities \cite{Hodge:Kleinian-singularities}. Motivated by Van den Bergh \cite{Van-den-Bergh:Koszul-bimodule-complex}, we give another point of view to obtain noncommutative generalized Weyl algebras by deforming commutative algebras. The main difference between Hodges's deformation and ours is that the Kleinian singularities, via Hodges's deformation, may become smooth; our deformation preserves the (non)smoothness.

We also show that the map $f\colon M\to\{\text{Per $2$-cocycles}\}$ induces an isomorphism $H_0(A,M^\nu)\cong H^2(A,M)$ if $A$ is homologically smooth, whose inverse is given explicitly. This isomorphism is a Van den Bergh duality.

This paper is organized as follows. In Sect.\ \ref{sec:preliminary}, besides reviewing the definitions of generalized Weyl algebra and formal deformation, we introduce homotopy double complexes as well as the associated total complexes. In Sect.\ \ref{sec:projective-resolution}, we construct a homotopy double complex for the generalized Weyl algebra $A=\kk[z;\lambda,\eta,\varphi(z)]$ and prove that the associated total complex is a periodic projective resolution. In Sect.\ \ref{sec:homological-smoothness}, using the periodic projective resolution, we prove that $A$ is homologically smooth if $\varphi(z)$ has no multiple roots, and furthermore it is twisted Calabi-Yau. In Sect.\ \ref{sec:noncommutative-deformation}, deformations of generalized Weyl algebras are studied. We construct formal deformations of noncommutative generalized Weyl algebras, and illustrate how to realize them by deforming commutative algebras, under a technical assumption. We also give an explicit Van den Bergh duality $H_0(A,M^\nu)\cong H^2(A,M)$ and explain that our construction is non-trivial in most cases.

\section{Preliminaries}\label{sec:preliminary}

Throughout, $\kk$ is a field, $\kk^\times=\kk\setminus\{0\}$, and all vector spaces and algebras are over $\kk$ unless stated otherwise. Unadorned $\otimes$ means $\otimes_{\kk}$. Let $A$ be an algebra and $M$ an $A$-bimodule. The group of algebra automorphisms of $A$ is denoted by $\Aut(A)$. For any $f$, $g\in\Aut(A)$, denote by ${}^f\!M^g$ the $A$-bimodule whose ground vector space is the same with $M$ and whose left and right $A$-actions are twisted by $f$ and $g$ respectively, that is, $a_1\cdot m\cdot a_2=f(a_1)mg(a_2)$ for any $a_1$, $a_2\in A$, $m\in M$. If one of $f$ and $g$ is the identity map, it is usually omitted.

Let $A^{\mathrm{op}}$ be the opposite algebra of $A$ and $A^e=A\otimes A^{\mathrm{op}}$ the enveloping algebra of $A$. An $A$-bimodule $M$ can be viewed as a left $A^e$-module in a natural way, that is, $(a_1\otimes a_2)\cdot m=a_1ma_2$ for any $a_1$, $a_2\in A$ and $m\in M$.

\subsection{Generalized Weyl algebras}\label{subsec:GWA-definition}

In this subsection, we recall the definition of generalized Weyl algebras given by Bavula in \cite{Bavula:GWA-def}.

\begin{defn}
Suppose $B$ is an algebra. For a central element $a\in B$ and an algebra automorphism $\sigma\in\Aut(B)$, the \textit{generalized Weyl algebra} (GWA for short) $A=B(\sigma,a)$ is by definition generated by $2$ variables $x$ and $y$ over $B$ subject to
\begin{gather*}
xb=\sigma(b)x,\;yb=\sigma^{-1}(b)y,\;\forall\, b\in B,\\
yx=a,\;xy=\sigma(a).
\end{gather*}
\end{defn}

Denote
\[
x_i=\begin{cases}
x^i, & \text{if $i\ge 0$},\\
y^{-i}, & \text{if $i< 0$},
\end{cases}
\]
then $A=\bigoplus_{i\in\mathbb{Z}}Bx_i$, and $Bx_i=x_iB$.

There are various algebras belong to the class of GWAs, such as the usual Weyl algebra $A_1(\kk)$, the enveloping algebra $U(\mathfrak{sl}(2,\kk))$ as well as its primitive factors $U(\mathfrak{sl}(2,\kk))/(C-\lambda)$ where $C$ is the Casimir element and $\lambda\in\kk$, the quantum $2$-spheres, and so on (see \cite{Bavula:GWA-tensor-product}).

Many properties of GWAs have been studied. But the literature on their homological smoothness is quite limited. Recall that an algebra is said to be \emph{homologically smooth} if as a bimodule over itself, it admits a finitely generated projective resolution of finite length. One aim of this paper is to study the homological smoothness and twisted Calabi-Yau property of a class of GWAs.

\begin{defn}[\cite{Bocklandt:superpotential}]
Suppose that $A$ is an algebra and $\nu\in\Aut(A)$. $A$ is called $\nu$-\textit{twisted Calabi-Yau} of dimension $d$ for some $d\in\mathbb{N}$ if $A$ is homologically smooth, and
\[
\Ext_{A^e}^i(A,A^e)\cong
\begin{cases}
0, & \text{if $i\neq d$},\\
A^{\nu}, & \text{if $i=d$}
\end{cases}
\]
as $A^e$-modules, where the left $A^e$-module structure of $A^e$ is used to compute the homology and the right one is retained, inducing the $A^e$-module structures on the homology groups.
\end{defn}

\begin{remark}
In the definition, the integer $d$ is equal to the Hochschild cohomological dimension of $A$. The automorphism $\nu$ is unique up to inner isomorphism and is thus called the \textit{Nakayama automorphism} of $A$ (see \cite{Brown-Zhang:noetherian-hopf-algebra}).
\end{remark}

\subsection{Spectral sequence of a homotopy double complex}
Let us introduce the notion of homotopy double complexes and the associated total complexes.

\begin{defn}
Suppose that $\mathcal{A}$ is an abelian category. Let $\{C^{pq}\}_{p,q\in\mathbb{Z}}$ be a family of objects in $\mathcal{A}$ together with morphisms $d_v$, $d_h$, $s$ of degrees $(0,1)$, $(1,0)$, $(2,-1)$ respectively.
The $4$-tuple $(C^{\cdot\cdot},d_v,d_h,s)$ is called a \textit{homotopy double cochain complex} if
\begin{equation}\label{eq:homotopy-double-complex}
d_v^2=0,\; d_hd_v+d_vd_h=0,\; d_h^2+d_vs+sd_v=0,\; d_hs+sd_h=0,\; s^2=0.
\end{equation}
The \textit{associated total complex} $(\Tot C^{\cdot\cdot},d)$ is defined by $(\Tot C^{\cdot\cdot})^n=\bigoplus_{p+q=n}C^{pq}$ and $d=d_v+d_h+s$.
\end{defn}

Homotopy double chain complexes and the associated total complexes can be defined similarly. We put in two pictures for the reader to visualize the definitions.
\[
\text{cochain}:\quad
\xymatrix@C=5mm@R=5mm{
C^{02} & C^{12} & C^{22} \\
C^{01} \ar[u]^-{d_v}\ar[r]^-{d_h}\ar[rrd]^-{s} & C^{11} & C^{21} \\
C^{00} & C^{10} & C^{20}
}
\qquad\qquad\text{chain}:\quad
\xymatrix@C=5mm@R=5mm{
C_{02} & C_{12} & C_{22} \\
C_{01} & C_{11} & C_{21} \ar[d]^-{d^v}\ar[l]^-{d^h}\ar[llu]^-{s} \\
C_{00} & C_{10} & C_{20}
}
\]

It is easy to see that by letting $s=0$, a homotopy double complex as well as the associated total complex is exactly the usual double complex as well as the usual total complex.

\begin{thm}\label{thm:proj.resolution}
  Suppose that there exist enough projective objects in $\mathcal{A}$. Let $C_{\cdot}$ be a chain complex, and $(\calP_{\cdot\cdot},d^v,d^h,r)$ be a homotopy double complex on the upper-half plane. Suppose that for each $p$, $\calP_{p,\cdot}$ is a projective resolution of $C_{p}$ and that the differentials of $C_\cdot$ are induced by $d^h$. Then $\Tot \calP_{\cdot\cdot}$ is quasi-isomorphic to $C_{\cdot}$ if one of the following conditions holds:
  \begin{enumerate}
  \item there is an integer $N$ such that $\calP_{pq}=0$ for all $q>N$,
  \item there is an integer $N$ such that $\calP_{pq}=0$ for all $p<N$.
  \end{enumerate}
\end{thm}

\begin{proof}
The filtration by columns
\[
(F_n\calP_{\cdot\cdot})_{pq}=
\begin{cases}
\calP_{pq}, & \text{if $p\le n$},\\
0, & \text{if $p>n$}
\end{cases}
\]
makes $\Tot \calP_{\cdot\cdot}$ into a filtered complex, and thus gives rise to a spectral sequence $E^r_{pq}$, starting with $E^0_{pq}=\calP_{pq}$. The differentials $d^0$ are just $d^v$, so $E^1_{pq}=H^v_q(\calP_{p,\cdot})$. The differentials $d^1$ are induced by $d^h$ since $(d^h)^2$ is null homotopy, so we have $E^2_{pq}=H^h_pH^v_q(\calP_{\cdot\cdot})$.

Since $\calP_{p,\cdot}$ is a projective resolution of $C_{p}$, we have
\[
E^2_{pq}=
\begin{cases}
  H_p(\calP_\cdot), &\text{if $q=0$},\\
  0, &\text{if $q\neq0$}.
\end{cases}
\]
In both cases, the filtration is bounded, thus we have a convergent spectral sequence $E_2^{pq}\Rightarrow H^{p+q}(\Tot \calP_{\cdot\cdot})$. Therefore, $\Tot \calP_{\cdot\cdot}$ is quasi-isomorphic to $C_{\cdot}$.
\end{proof}

\begin{remark}
Unlike the usual double complexes, one fails to endow the total complex of $\calP_{\cdot\cdot}$ with the filtration by rows because the differentials do not respect the filtration.
\end{remark}

\subsection{Deformations of an associative algebra}\label{subsec:deformation-survey}
The first part of this subsection is devoted to a review of formal deformations. The reader is referred to the survey \cite{Gerstenhaber-Schack:deformation} for details. In the second part, we introduce locally finite deformations.

Let $A$ be an algebra, and $(C^\cdot(A,A), \mathsf{b})$ the Hochschild cochain complex of $A$. Denote by $\kk[[t]]$ the ring of formal power series in an indeterminate $t$, and by $A[[t]]$ the $\kk[[t]]$-module of formal power series $\sum_{n=0}^{\infty}a_nt^n$ with coefficients in $A$. Given a family of $\kk$-bilinear maps $F_n\colon A\times A\to A$, $n\geq 1$, one obtains a $\kk$-bilinear map $*\colon A\times A\to A[[t]]$ defined by
\begin{equation}\label{eq:formal-deformation}
u*v=uv+F_1(u,v)t+F_2(u,v)t^2+\cdots.
\end{equation}
A \textit{formal deformation} of $A$ is such a $*$ that the extended $\kk[[t]]$-bilinear map $A[[t]]\times A[[t]]\to A[[t]]$ determines an associative multiplication on $A[[t]]$. In this case, the maps $F_n$ satisfy
\begin{equation}\label{eq:obstruction}
\sum_{i=1}^{n-1}F_i\bullet F_{n-i}=\mathsf{b}F_n
\end{equation}
where $F_i\bullet F_{n-i}\in C^3(A,A)$ is defined by
\[
F_i\bullet F_{n-i}(a_1,a_2,a_3)=F_i(F_{n-i}(a_1,a_2),a_3)-F_i(a_1,F_{n-i}(a_2,a_3)).
\]

Two formal deformations $*$ and $*'$ are said to be \textit{equivalent} if there is a $\kk[[t]]$-algebra isomorphism $G \colon (A[[t]],*)\to(A[[t]],*')$ such that
\[
G(u)\equiv u \;\;\bmod {tA[[t]]}
\]
for all $u\in A$. We use the symbol $\cong_\mathrm{f}$ to express the kind of isomorphisms.

In the definition of formal deformation, by replacing $\kk[[t]]$, $A[[t]]$ by $\kk[t]/(t^{n+1})$, $A\otimes\kk[t]/(t^{n+1})$ respectively for $n\in \mathbb{Z}^+$, the \textit{$n$th order deformation} and equivalence relation can be defined similarly.

If we view each $F_n$ as an element in the Hochschild cochain module $C^2(A,A)$, then $F_1$ must be a $2$-cocycle by the associative law. Moreover, let $F_1'$ be another $2$-cocycle.
Then $F_1$, $F_1'$ represent the same cohomology class in $H^2(A,A)$ if and only if $F_1'$ appears in a formal deformation $*'$ equivalent to $*$. In fact, there is a bijection between $H^2(A,A)$ and the family of equivalence classes of first order deformations. A natural question is: Is any $2$-cocycle able to lift to a formal deformation? The answer is no in general.

Let $F_0$ be the multiplication map of $A$. Starting with a certain $F_1\in Z^2(A,A)$, $F_0+F_1t$ defines an associative multiplication on $A\otimes\kk[t]/(t^2)$. One can show $F_1\bullet F_1\in Z^3(A,A)$, and by \eqref{eq:obstruction}, $F_1\bullet F_1=\mathsf{b}F_2$ for some $F_2$ if and only if $F_0+F_1t+F_2t^2$ defines an associative multiplication on $A\otimes\kk[t]/(t^3)$. The cohomology class $[F_1\bullet F_1]\in H^3(A,A)$ is vividly called the \textit{primary obstruction} to integrating $F_1$. Generally, if there is an $(n-1)$st order deformation $F_0+F_1t+\cdots+F_{n-1}t^{n-1}$, then the left-hand side of \eqref{eq:obstruction} is always a Hochschild $3$-cocycle whose cohomology class is called an obstruction, and it is a coboundary if and only if there exists an $n$th order deformation $F_0+F_1t+\cdots+F_{n-1}t^{n-1}+F_nt^n$. If all obstructions can be passed successfully, i.e., $[F_1\bullet F_{n}+F_2\bullet F_{n-1}+\cdots+F_{n}\bullet F_1]=0$ in $H^3(A,A)$ for all $n\geq 1$, we say that $F_1$ is \textit{integrable} and $F_2, F_3, \ldots$ integrate $F_1$.

It is a difficult problem to decide when a $2$-cocycle is integrable, unless luckily, one has $H^3(A,A)=0$. We will show that the homologically smooth GWAs studied in Sect.\ \ref{sec:homological-smoothness} are such algebras.

Let $*$ be a formal deformation of $A$ and $\tilde{A}=A\otimes \kk[[t]]$. Since $\kk$ is a field, $\tilde{A}$ is a $\kk[[t]]$-submodule of $A[[t]]$. It is easy to verify that $\tilde{A}$ is a $\kk[[t]]$-subalgebra of $(A[[t]], *)$ if and only if the right-hand side of \eqref{eq:formal-deformation} belongs to $\tilde{A}$ for all $u$, $v\in A$. In this case, we say $*$ to be \textit{locally finite}, i.e., the vector space $\sum_{n\geq 1}\kk F_n(u,v)$ is finite dimensional for every pair $(u,v)\in A\times A$. Moreover, the $t$-adic completion of $(\tilde{A},*|_{\tilde{A}})$ is isomorphic to $(A[[t]],*)$.  Let $\tilde{A}_t$ be the localization of $\tilde{A}$ at $t$, which is a $\kk(\!(t)\!)$-algebra. We call $\tilde{A}_t$ a \emph{locally finite deformation} of $A$.

The trivial formal deformation $*_{\mathrm{tr}}$ is given by $F_n=0$ for all $n\geq 1$, i.e., $(A[[t]],*_{\mathrm{tr}})$ is the algebra of the formal power series with coefficients in $A$, and $\tilde{A}_t\cong \kk(\!(t)\!)\otimes A$ is an extension of base field.

\section{Projective resolutions}\label{sec:projective-resolution}

\subsection{A periodic complex}
For any GWA $A=B(\sigma, a)$, we construct a periodic complex $C_{\cdot}$ of $A^e$-modules as the cornerstone of the conclusions in this paper.

\begin{prop}\label{prop:long-sequence}
Suppose that $A=B(\sigma, a)$ is a GWA. Let $C_{\cdot}\in\mathrm{Ch}_{\geq 0}(A^e)$ be the chain complex of $A^e$-modules with $C_0=A\otimes_BA$, $C_{i}=(A^{\sigma}\otimes_BA)\oplus(A\otimes_B{}^{\sigma}\!A)$ for all odd $i$, and $C_{i}=(A\otimes_BA)\oplus(A\otimes_B A)$ for all even $i>0$, whose differentials $d_{i}\colon C_i\to C_{i-1}$ are defined by
\begin{alignat*}{2}
& d_{1}(1\otimes 1, 0)=x\otimes 1-1\otimes x, & \quad & d_{1}(0, 1\otimes 1)=y\otimes 1-1\otimes y,\\
\intertext{for all $j>0$,}
& d_{2j}(1\otimes 1, 0)=(y\otimes 1, 1\otimes x), & \quad & d_{2j}(0,1\otimes 1)=(1\otimes y, x\otimes 1),\\
& d_{2j+1}(1\otimes 1, 0)=(x\otimes 1, -1\otimes x), & \quad & d_{2j+1}(0,1\otimes 1)=(-1\otimes y, y\otimes 1).
\end{alignat*}
If $a\in B$ is not a zero-divisor, then $H_i(C)=0$ for all $i\neq 0$ and $H_0(C)=A$.
\end{prop}

\begin{proof}
It is routine to check that $d_i$'s are well defined and are differentials.

It is easy to see $H_0(C)=A$. We need to prove the exactness in degrees 1, 2, 3.

Since $A$ is a free left and right $B$-module with a basis $\{x_i\}_{i\in\mathbb{Z}}$, $A^{\sigma}\otimes_BA$, $A\otimes_B{}^{\sigma}\!A$, and $A\otimes_BA$ are all isomorphic $\bigoplus_{i,\,j\in\mathbb{Z}}x_iB x_j$ as vector spaces. Thus any element in them can be expressed to be $\sum_{i,\,j}x_i\otimes b_{ij}x_j$ uniquely with $b_{ij}\in B$. Such an expression is called standard.

(1) $\Ker d_{1}\subset\Ima d_{2}$: For any $P=(\sum_{i,\,j}x_i\otimes b_{ij}x_j,\sum_{p,\,q}x_p\otimes c_{pq}x_q)\in\Ker d_{1}$, we have
\begin{equation}\label{eq:exact0}
\sum_{i,\,j}x_ix\otimes b_{ij}x_j-\sum_{i,\,j}x_i\otimes xb_{ij}x_j=-\sum_{p,\,q}x_py\otimes c_{pq}x_q+\sum_{p,\,q}x_p\otimes yc_{pq}x_q.
\end{equation}
The standard forms of the four summations in \eqref{eq:exact0} are
\begin{align*}
  &\sum_{i\geq 0}\sum_jx_{i+1}\otimes b_{ij}x_j+\sum_{i\leq -1}\sum_jx_{i+1}\otimes ab_{ij}x_j,\\
  &\sum_{i}\sum_{j\geq 0}x_i\otimes \sigma(b_{ij})x_{j+1}+\sum_{i}\sum_{j\leq -1}x_i\otimes \sigma(b_{ij})\sigma(a)x_{j+1},\\
  &\sum_{p\leq 0}\sum_q x_{p-1}\otimes c_{pq}x_q+\sum_{p\geq 1}\sum_q x_{p-1}\otimes \sigma(a)c_{pq}x_q,\\
  &\sum_{p}\sum_{q\leq 0}x_p\otimes \sigma^{-1}(c_{pq})x_{q-1}+\sum_{p}\sum_{q\geq 1}x_p\otimes \sigma^{-1}(c_{pq})ax_{q-1}.
\end{align*}
Endow $\mathbb{Z}\times\mathbb{Z}$ with the lexicographic order. Let $(i',j')$ and $(p',q')$ be the greatest indexes such that $b_{i'j'}\neq 0$, $c_{p'q'}\neq 0$, respectively. Since $a\neq 0$, by observing the four standard forms, the highest nonzero term of the left-hand side in \eqref{eq:exact0}, $x_{i'+1}\otimes b_{i'j'}x_{j'}$ or $x_{i'+1}\otimes ab_{i'j'}x_{j'}$, must cancel that of the right-hand side, $x_{p'}\otimes \sigma^{-1}(c_{p'q'})x_{q'-1}$ or $x_{p'}\otimes \sigma^{-1}(c_{p'q'})ax_{q'-1}$. This is equivalent to
\[
x_{i'}x\otimes b_{i'j'}x_{j'}= x_{p'}\otimes yc_{p'q'}x_{q'}.
\]
It follows that $(i'+1,j')=(p',q'-1)$, and
\begin{equation}\label{eq:exact1}
x_{i'}x\otimes b_{i'j'}x_{j'}= x_{i'+1}\otimes \sigma^{-1}(c_{p'q'})yx_{j'+1}.
\end{equation}

We have the following four cases.

(i) $i'\ge 0$, $j'\le -1$. Equation \eqref{eq:exact1} becomes $x^{i'+1}\otimes b_{i'j'}y^{-j'}= x^{i'+1}\otimes \sigma^{-1}(c_{p'q'})y^{-j'}$. So $b_{i'j'}=\sigma^{-1}(c_{p'q'})$, and
\begin{align*}
&\varphantom{=}(x_{i'}\otimes b_{i'j'}x_{j'},x_{p'}\otimes c_{p'q'}x_{q'})=x^{i'}(1\otimes \sigma^{-1}(c_{p'q'})y, x\otimes c_{p'q'})y^{-q'}\\
&=x^{i'}(1\otimes yc_{p'q'}, x\otimes c_{p'q'})y^{-q'}=x^{i'}(1\otimes y, x\otimes 1)c_{p'q'}y^{-q'}\\
&=d_2(0, x^{i'}\otimes c_{p'q'}y^{-q'}).
\end{align*}

(ii) $i'\le -1$, $j'\le -1$. Equation \eqref{eq:exact1} becomes $y^{-i'-1}\otimes ab_{i'j'}y^{-j'}= y^{-i'-1}\otimes \sigma^{-1}(c_{p'q'})y^{-j'}$. So $c_{p'q'}=\sigma(a)\sigma(b_{i'j'})=xy\sigma(b_{i'j'})=xb_{i'j'}y$, and
\begin{align*}
&\varphantom{=}(x_{i'}\otimes b_{i'j'}x_{j'},x_{p'}\otimes c_{p'q'}x_{q'})=y^{-p'}(y\otimes b_{i'j'}y, 1\otimes xb_{i'j'}y)y^{-j'-1}\\
&=y^{-p'}(y\otimes 1, 1\otimes x)b_{i'j'}y^{-j'}=d_2(y^{-p'}\otimes b_{i'j'}y^{-j'},0).
\end{align*}

(iii) $i'\ge 0$, $j'\ge 0$. Equation \eqref{eq:exact1} becomes $x^{i'+1}\otimes b_{i'j'}x^{j'}= x^{i'+1}\otimes \sigma^{-1}(c_{p'q'})ax^{j'}$. So $b_{i'j'}=a\sigma^{-1}(c_{p'q'})=yx\sigma^{-1}(c_{p'q'})=yc_{p'q'}x$, and
\begin{align*}
&\varphantom{=}(x_{i'}\otimes b_{i'j'}x_{j'},x_{p'}\otimes c_{p'q'}x_{q'})=x^{i'}(1\otimes yc_{p'q'}x, x\otimes c_{p'q'}x)x^{j'}\\
&=x^{i'}(1\otimes y, x\otimes 1)c_{p'q'}x^{q'}=d_2(0, x^{i'}\otimes c_{p'q'}x^{q'}).
\end{align*}

(iv) $i'\le -1$, $j'\ge 0$. Equation \eqref{eq:exact1} becomes $y^{-i'-1}\otimes a b_{i'j'}x^{j'}= y^{-i'-1}\otimes \sigma^{-1}(c_{p'q'})ax^{j'}$. So $c_{p'q'}=\sigma(b_{i'j'})$ since $a$ is not a zero-divisor. Thus
\begin{align*}
&\varphantom{=}(x_{i'}\otimes b_{i'j'}x_{j'},x_{p'}\otimes c_{p'q'}x_{q'})=y^{-p'}(y\otimes b_{i'j'}, 1\otimes \sigma(b_{i'j'})x)x^{j'}\\
&=y^{-p'}(y\otimes b_{i'j'}, 1\otimes xb_{i'j'})x^{j'}=y^{-p'}(y\otimes 1, 1\otimes x)b_{i'j'}x^{j'}\\
&=d_2(y^{-p'}\otimes b_{i'j'}x^{j'}, 0).
\end{align*}

In each case, the highest term of $P$ lies in $\Ima d_{2}$. By considering the next highest pair
in the lexicographic order and repeating the above argument, we have $P\in\Ima d_{2}$.

(2) $\Ker d_{2}\subset\Ima d_{3}$: For any $P=(\sum_{i,\,j}x_i\otimes b_{ij}x_j,\sum_{p,\,q}x_p\otimes c_{pq}x_q)\in\Ker d_{2}$, we have
\[
\left\{
\begin{alignedat}{2}
\sum_{i,\,j}x_iy\otimes b_{ij}x_j+\sum_{p,\,q}x_p\otimes yc_{pq}x_q&=0\\
\sum_{i,\,j}x_i\otimes xb_{ij}x_j+\sum_{p,\,q}x_px\otimes c_{pq}x_q&=0.
\end{alignedat}
\right.
\]

Let $(i',j')$ and $(p',q')$ be the indexes of the highest nonzero terms in both coordinates of $P$. By a similar argument with (1), we have
\[
\left\{
\begin{alignedat}{2}
x_{i'}y\otimes b_{i'j'}x_{j'}+x_{p'}\otimes yc_{p'q'}x_{q'}&=0\\
x_{i'}\otimes xb_{i'j'}x_{j'}+x_{p'}x\otimes c_{p'q'}x_{q'}&=0.
\end{alignedat}
\right.
\]
Hence $(i'-1,j')=(p',q'-1)$ and
\begin{align}
  x_{i'}y\otimes b_{i'j'}x_{j'}+x_{p'}\otimes \sigma^{-1}(c_{p'q'})yx_{q'}&=0,\label{eq:exact2}\\
  x_{i'}\otimes \sigma(b_{i'j'})xx_{j'}+x_{p'}x\otimes c_{p'q'}x_{q'}&=0.\label{eq:exact3}
\end{align}

Consider the following four cases: (i) $i'\geq 1$, $j'\geq 0$, (ii) $i'\geq 1$, $j'\leq -1$, (iii) $i'\leq 0$, $j'\geq 0$, (iv) $i'\leq 0$, $j'\leq -1$. In each case, we obtain the results from \eqref{eq:exact2}, \eqref{eq:exact3} as follows.
\begin{enumerate}[(i)]
\item $c_{p'q'}=-\sigma(b_{i'j'})$ and $(x_{i'}\otimes b_{i'j'}x_{j'}, x_{p'}\otimes c_{p'q'}x_{q'})=x^{p'}(x\otimes 1, -1\otimes x)b_{i'j'}x^{j'}=d_3(x^{p'}\otimes b_{i'j'}x^{j'}, 0)$.
\item $c_{p'q'}=-\sigma(a)\sigma(b_{i'j'})=-xb_{i'j'}y$ and $(x_{i'}\otimes b_{i'j'}x_{j'}, x_{p'}\otimes c_{p'q'}x_{q'})=x^{p'}(x\otimes 1, -1\otimes x)b_{i'j'}y^{-j'}=d_3(x^{p'}\otimes b_{i'j'}y^{-j'}, 0)$.
\item $b_{i'j'}=-\sigma^{-1}(c_{p'q'})a=-yc_{p'q'}x$ and $(x_{i'}\otimes b_{i'j'}x_{j'}, x_{p'}\otimes c_{p'q'}x_{q'})=y^{-i'}(-1\otimes y, y\otimes 1)c_{p'q'}x^{q'}=d_3(0, y^{-i'}\otimes c_{p'q'}x^{q'})$.
\item $b_{i'j'}=-\sigma^{-1}(c_{p'q'})$ and $(x_{i'}\otimes b_{i'j'}x_{j'}, x_{p'}\otimes c_{p'q'}x_{q'})=y^{-i'}(-1\otimes y, y\otimes 1)c_{p'q'}y^{-q'}=d_3(0, y^{-i'}\otimes c_{p'q'}y^{-q'})$.
\end{enumerate}
Just as what we did in (1), we have $P\in \Ima d_{3}$.

(3) $\Ker d_{3}\subset\Ima d_{4}$: This case is similar with (2), so we omit the proof.
\end{proof}

\subsection{Construction of homotopy double complex}

As stated in Theorem \ref{thm:proj.resolution}, a projective resolution of the complex $C_{\cdot}$ can be constructed if there is a suitable homotopy double complex $(\calP_{\cdot\cdot},d^v,d^h,r)$ on the upper-half plane. In particular, applying Theorem \ref{thm:proj.resolution} to Proposition \ref{prop:long-sequence} if such $\calP_{\cdot\cdot}$ exists, $\Tot \calP_{\cdot\cdot}$ is an $A^e$-projective resolution of $A$.

The question is how to construct a homotopy double complex from the periodic complex $C_\cdot$.  Note that we have $C_0=A\otimes_BA\cong A\otimes_B B\otimes_B A$, and that we have similar presentations for all the $C_p$. Note also that each of the functors $A\otimes_B-$ and $-\otimes_B A$ are exact, since $A$ is flat (free) over $B$. So from a projective bimodule resolution $\mathcal{K}_\cdot$ of $B$ we get a resolution $\calP_{0,\cdot}=A\otimes_B \mathcal{K}_\cdot\otimes_B A$ of $C_0$ whose differentials are denoted by $d^v_{0,\cdot}$. Projectivity of the new resolution $\calP_{0,\cdot}$ over $A^e$ follows form the projectivity of $\mathcal{K}_\cdot$ over $B^e$. By a similar process, we get resolutions $\calP_{p,\cdot}$ of all the $C_p$. Namely, we obtain a family $\{\calP_{pq}\}_{p,q}$ of projective $A^e$-modules as well as differentials $d^v$ of degree $(0,-1)$.

Next, due to the Comparison Lemma, each differential $d_p\colon C_p\to C_{p-1}$ lifts to a morphism $\calP_{p,\cdot}\to \calP_{p-1,\cdot}$ of complexes. Using the sign's trick, we obtain $d^h_{pq}\colon \calP_{pq}\to \calP_{p-1,q}$ such that the second equation in \eqref{eq:homotopy-double-complex} is satisfied. It follows from $d_{p-1}d_p=0$ that $d^h_{p-1,\cdot}d^h_{p,\cdot}$ is null homotopic. This in turn indicates the existence of homotopy $r_{pq}\colon \calP_{pq}\to \calP_{p-2,q+1}$ such that the third equation in \eqref{eq:homotopy-double-complex} is also satisfied (Here we use $r$ instead of $s$). If the last two hold also, then $(\calP_{\cdot\cdot},d^v,d^h,r)$ is a homotopy double complex as required.

We will not discuss the existence of homotopy double complexes for a general GWA. Instead, let us restrict our attention to the special case: $B=\kk[z]$. In this case, $a$ is not a zero-divisor if and only if $a\neq 0$. Let $a=\varphi(z)=\sum_{i=0}^{l}a_iz^i$ with $a_l\neq 0$. Since $\sigma(z)$ must be of the form $\lambda z+\eta$ for some $\lambda\in\kk^\times$, $\eta\in\kk$, we write the GWA $A=B(\sigma,a)$ as $\kk[z;\lambda,\eta,\varphi(z)]$.

Choose bimodule projective resolutions of $\kk[z]$ to be
\[
0\to \kk[z]\otimes \kk[z]\xrightarrow{\delta_c} \kk[z]\otimes \kk[z] \rightarrow \kk[z] \to 0
\]
where $\delta_c(1\otimes 1)=c(z\otimes 1-1\otimes z)$ for any $c\in\kk^\times$. According to the construction of $P_{\cdot\cdot}$ given before, we have $\calP_{00}=\calP_{01}=A\otimes A$, $\calP_{10}=\calP_{11}=(A^\sigma\otimes A)\oplus (A\otimes{}^\sigma\! A)$, and so on. Note that as $A$-bimodules, $A^\sigma\otimes A=A\otimes{}^\sigma\! A=A\otimes A$, since the actions are all given by
 \[ a\triangleright(u\otimes v)\triangleleft b=au\otimes vb.\]
So we have $\calP_{j0}=\calP_{j1}=(A\otimes A)^{\oplus 2}$ for $j\geq 1$. The nonzero vertical differentials are only $d^v_{p1}$ for all $p\geq 0$, which are written as $d^v_p$ for short. Let $d^v_0$ be induced by $\delta_1$, $d^v_j$ be induced by $(\delta_1,\delta_{\lambda^{-1}})$ for all odd $j$, and by $(\delta_1,\delta_1)$ for all positive even $j$. Explicitly, $d^v_0(1\otimes 1)=z\otimes 1-1\otimes z$,
\begin{align*}
  d^v_1(1\otimes 1, 0)&=(\sigma(z)\otimes 1-1\otimes z, 0), & d^v_1(0, 1\otimes 1)&=(0, \sigma^{-1}(z)\otimes 1-1\otimes z),\\
  d^v_2(1\otimes 1, 0)&=(z\otimes 1-1\otimes z, 0), & d^v_2(0, 1\otimes 1)&=(0, z\otimes 1-1\otimes z),
\end{align*}
and so on.

By the Comparison Lemma and the sign's trick, we define $d^h_{\cdot,0}$, $d^h_{\cdot,1}$ as follows,
\begin{alignat*}{2}
& d^h_{10}(1\otimes 1, 0)=x\otimes 1-1\otimes x, & \quad & d^h_{10}(0, 1\otimes 1)=y\otimes 1-1\otimes y,\\
& d^h_{2j,0}(1\otimes 1, 0)=(y\otimes 1, 1\otimes x), & \quad & d^h_{2j,0}(0,1\otimes 1)=(1\otimes y, x\otimes 1),\\
& d^h_{2j+1,0}(1\otimes 1, 0)=(x\otimes 1, -1\otimes x), & \quad & d^h_{2j+1,0}(0,1\otimes 1)=(-1\otimes y, y\otimes 1),\\
& d^h_{11}(1\otimes 1,0)=-x\otimes 1+\lambda\otimes x, &\quad& d^h_{11}(0,1\otimes 1)=-y\otimes 1+\lambda^{-1}\otimes y,\\
& d^h_{2j,1}(1\otimes 1,0)=(-y\otimes 1,-\lambda\otimes  x), &\quad& d^h_{2j,1}(0,1\otimes 1)=(-\lambda^{-1}\otimes y, -x\otimes 1),\\
& d^h_{2j+1,1}(1\otimes 1,0)=(-x\otimes 1, \lambda \otimes x), &\quad& d^h_{2j+1,1}(0,1\otimes 1)=(\lambda^{-1}\otimes y, -y\otimes 1).
\end{alignat*}

After that, let us construct $r_{p}=r_{p0}\colon \calP_{p0}\to\calP_{p-2,1}$. Define the linear map $\Delta_0\colon \kk[z]\to \kk[z]\otimes\kk[z]$ by $\Delta_0(1)=0$ and $\Delta_0(z^k)=\sum_{i=1}^{k}z^{k-i}\otimes z^{i-1}$ for $k\ge 1$. Let $\iota\colon\kk[z]\hookrightarrow A$ be the natural embedding. For any $\kk$-linear endomorphisms $f$, $g$ of $\kk[z]$, we denote $(\iota\otimes\iota)\circ(f\otimes g)\circ\Delta_0$ by ${}^f\!\Delta^g$, and usually suppress $f$ or $g$ if it is the identity map. By a direct computation, we have
\begin{align*}
d^h_{10}d^h_{20}(1\otimes 1, 0)&=\varphi(z)\otimes 1-1\otimes \varphi(z),\\
d^h_{10}d^h_{20}(0, 1\otimes 1)&=\sigma(\varphi(z))\otimes 1-1\otimes \sigma(\varphi(z)),\\
d^h_{2j,0}d^h_{2j+1,0}(1\otimes 1, 0)&=(\sigma(\varphi(z))\otimes 1-1\otimes \varphi(z), 0),\\
d^h_{2j,0}d^h_{2j+1,0}(0, 1\otimes 1)&=(0, \varphi(z)\otimes 1-1\otimes \sigma(\varphi(z))),\\
d^h_{2j+1,0}d^h_{2j+2,0}(1\otimes 1, 0)&=(\varphi(z)\otimes 1-1\otimes \varphi(z), 0),\\
d^h_{2j+1,0}d^h_{2j+2,0}(0, 1\otimes 1)&=(0, \sigma(\varphi(z))\otimes 1-1\otimes \sigma(\varphi(z))).
\end{align*}
Thus $r$ are defined by
\begin{alignat*}{2}
& r_{2}(1\otimes 1, 0)=-\Delta(\varphi(z)), &\quad& r_{2}(0,1\otimes 1)=-\lambda{}^{\sigma}\!\Delta^{\sigma}(\varphi(z)),\\
& r_{2j+1}(1\otimes 1, 0)=(-{}^{\sigma}\!\Delta(\varphi(z)), 0), &\quad& r_{2j+1}(0,1\otimes 1)=(0, -\lambda\Delta^{\sigma}(\varphi(z))),\\
& r_{2j+2}(1\otimes 1, 0)=(-\Delta(\varphi(z)), 0), &\quad& r_{2j+2}(0,1\otimes 1)=(0, -\lambda{}^{\sigma}\!\Delta^{\sigma}(\varphi(z))).
\end{alignat*}

\begin{prop}\label{prop:proj.dim=1}
Let $A=\kk[z;\lambda,\eta,\varphi(z)]$. The above formulas make $(\calP_{\cdot\cdot},d^v,d^h,r)$
\[
\xymatrix@C=10mm@R=12mm{
\calP_{01} \ar[d]^(0.4){d^v_0} & \calP_{11} \ar[l]_-{d^h_{11}}\ar[d]^(0.4){d^v_1} & \calP_{21} \ar[l]_-{d^h_{21}}\ar[d]^(0.4){d^v_2} & \calP_{31} \ar[l]_-{d^h_{31}}\ar[d]^(0.4){d^v_3} & \calP_{41} \ar[l]_-{d^h_{41}}\ar[d]^(0.4){d^v_4} & \cdots\hphantom{.} \ar[l] \\
\calP_{00} & \calP_{10} \ar[l]^-{d^h_{10}} & \calP_{20} \ar[l]^-{d^h_{20}}\ar[llu]^(0.6){r_2}|!{[lu];[l]}\hole & \calP_{30} \ar[l]^-{d^h_{30}}\ar[llu]^(0.6){r_3}|!{[lu];[l]}\hole & \calP_{40} \ar[l]^-{d^h_{40}}\ar[llu]^(0.6){r_4}|!{[lu];[l]}\hole & \cdots \ar[l]
}
\]
into a homotopy double complex, and thus $\Tot \calP_{\cdot\cdot}$ is an $A^e$-projective resolution of $A$.
\end{prop}

\begin{proof}
It suffices to verify the last two equations of \eqref{eq:homotopy-double-complex}.

By the definition of $r$, we have $d^h_{p-1,0}d^h_{p0}=-d^v_{p-2}r_p$ for all $p$. Then
\[
d^v_{p-3}d^h_{p-2,1}r_p=-d^h_{p-2,0}d^v_{p-2}r_p=d^h_{p-2,0}d^h_{p-1,0}d^h_{p0}=-d^v_{p-3}r_{p-1}d^h_{p0}.
\]
Since $d^v_{p-3}$ is injective, we have $d^h_{p-2,1}r_p+r_{p-1}d^h_{p0}=0$ for all $p$, which expresses that the fourth equation of \eqref{eq:homotopy-double-complex} is fulfilled. Since $r_p$ are only nonzero homotopy, the fifth equation holds trivially.
\end{proof}

\begin{remark}
Following \cite{Richard-Solotar:isom-quantum-GWA}, the GWA $\kk[z;\lambda,\eta,\varphi(z)]$ is called \textit{classical} if $\lambda=1$, $\eta\neq 0$, called \textit{quantum} if $\lambda\neq 1$, $\eta=0$.
In both cases, $A^e$-projective resolutions of $A$ are constructed in \cite{Farinati:Hochschid-homology-GWA} and \cite{Solotar:Hochschild-homology-GWA-quantum} respectively, via Smith algebra \cite{Smith:Smith-algebra}, i.e., the algebra $B_3$ given in Sect.\ \ref{sec:homological-smoothness}. Our results coincide with theirs, but in a different way.
\end{remark}

\section{Homological smoothness}\label{sec:homological-smoothness}
In this section, let $A=\kk[z;\lambda,\eta,\varphi(z)]$.

Using the homotopy double complex $\calP_{\cdot\cdot}$, the Hochschild cohomology of $A$ with coefficients in an $A$-bimodule $M$ can be computed. In fact, let $\calQ^{pq}=\Hom_{A^e}(\calP_{pq},M)$, and $\partial_v^{pq}$, $\partial_h^{pq}$, $s^{pq}$ be the maps obtained by letting $\Hom_{A^e}(-,M)$ act on $d^v_{p,q+1}$, $d^h_{p+1,q}$, $r_{p+2,q-1}$ respectively. Then $(\calQ^{\cdot\cdot},\partial_v,\partial_h,s)$ is also a homotopy double complex and
\[
H^n(A,M)=H^{n}(\Hom_{A^e}(\Tot\calP_{\cdot\cdot},M))= H^{n}(\Tot\calQ^{\cdot\cdot}).
\]
Since the complex $\Tot\calQ^{\cdot\cdot}$ is periodic, an element in $Z^n(\Tot\calQ^{\cdot\cdot})$ (resp.\ $B^n(\Tot\calQ^{\cdot\cdot})$) is called a Per $n$-cocycle (resp.\ Per $n$-coboundary).

Concretely, $\calQ^{\cdot\cdot}$ is given as follows,
\[
\xymatrix@R=10mm@C=7mm{
M \ar[r]^-{\partial_h^{01}}\ar[rrd]^(0.6){s^0}|!{[r];[rd]}\hole  & M\oplus M \ar[r]^-{\partial_h^{11}}\ar[rrd]^(0.6){s^1}|!{[r];[rd]}\hole & M\oplus M \ar[r]^-{\partial_h^{21}}\ar[rrd]^(0.6){s^2}|!{[r];[rd]}\hole & M\oplus M \ar[r]^-{\partial_h^{31}} & M\oplus M \ar[r] & \cdots \\
M \ar[r]_-{\partial_h^{00}}\ar[u]^(0.4){\partial_v^{0}}  & M\oplus M \ar[r]_-{\partial_h^{10}}\ar[u]^(0.4){\partial_v^{1}} & M\oplus M \ar[r]_-{\partial_h^{20}}\ar[u]^(0.4){\partial_v^{2}} & M\oplus M \ar[r]_-{\partial_h^{30}}\ar[u]^(0.4){\partial_v^{3}} & M\oplus M \ar[r]\ar[u]^(0.4){\partial_v^{4}} & \cdots
}
\]
where for all $j\geq 1$,
\begin{align*}
\partial_h^{00}(m)&=(xm-mx,ym-my),\\
\partial_h^{2j-1,0}(m_1,m_2)&=(ym_1+m_2x,m_1y+xm_2),\\
\partial_h^{2j,0}(m_1,m_2)&=(xm_1-m_2x,-m_1y+ym_2),\\
\partial_h^{01}(m)&=(-xm+\lambda mx,-ym+\lambda^{-1}my),\\
\partial_h^{2j-1,1}(m_1,m_2)&=(-ym_1-\lambda m_2x,-\lambda^{-1}m_1y-xm_2),\\
\partial_h^{2j,1}(m_1,m_2)&=(-xm_1+\lambda m_2x, \lambda^{-1}m_1y-ym_2),\\
\partial_v^{0}(m)&=zm-mz,\\
\partial_v^{2j-1}(m_1,m_2)&=(\sigma(z)m_1-m_1z,\lambda^{-1}zm_2-\lambda^{-1}m_2\sigma(z)),\\
\partial_v^{2j}(m_1,m_2)&=(zm_1-m_1z,\lambda^{-1}\sigma(z)m_2-\lambda^{-1}m_2\sigma(z)),\\
s^{0}(m)&=(-\Delta(\varphi)\cdot m,-\lambda{}^{\sigma}\!\Delta^{\sigma}(\varphi)\cdot m),\\
s^{2j-1}(m_1, m_2)&=(-{}^{\sigma}\!\Delta(\varphi)\cdot m_1, -\lambda\Delta^{\sigma}(\varphi)\cdot m_2),\\
s^{2j}(m_1, m_2)&=(-\Delta(\varphi)\cdot m_1, -\lambda{}^{\sigma}\!\Delta^{\sigma}(\varphi)\cdot m_2).
\end{align*}

Let $\partial^{\cdot}$ be the differentials of $\Tot\calQ^{\cdot\cdot}$. Denote by $\varphi'(z)$ the formal derivative of $\varphi(z)$.

\begin{lem}\label{lem:no-multi-root}
If $\varphi(z)$ has no multiple roots, then $H^{3}(\Tot\calQ^{\cdot\cdot})=0$ for all $A$-bimodules $M$.
\end{lem}

\begin{proof}
During the proof, we sometimes write a polynomial $h(z)$ as $h$, for simplicity. Fix $\alpha(z)$, $\beta(z)\in\kk[z]$ such that $\alpha(z)\varphi(z)+\beta(z)\varphi'(z)=1$.

Suppose that $(m_1,m_2,m_3,m_4)$ is a Per $3$-cocycle. Then $\partial_h^{21}(m_1,m_2)+\partial_v^{3}(m_3,m_4)=0$ and $s^{2}(m_1,m_2)+\partial_h^{30}(m_3,m_4)=0$, that is,
\begin{align}
-xm_1+\lambda m_2x+\sigma(z)m_3-m_3z&=0,\label{eq:homological-smooth-1}\\
m_1y-\lambda y m_2+zm_4-m_4\sigma(z)&=0,\label{eq:homological-smooth-2}\\
-\Delta(\varphi)\cdot m_1+ym_3+m_4x&=0,\label{eq:homological-smooth-3}\\
-\lambda{}^{\sigma}\!\Delta^{\sigma}(\varphi)\cdot m_2+m_3y+xm_4&=0.\label{eq:homological-smooth-4}
\end{align}
By induction, we obtain from \eqref{eq:homological-smooth-1}, \eqref{eq:homological-smooth-2} that for any $j\ge 1$,
\begin{align*}
\sigma(z)^jm_3-m_3z^j&=x(\Delta(z^{j})\cdot m_1)- \lambda({}^{\sigma}\!\Delta^{\sigma}(z^{j})\cdot m_2)x,\\
m_4\sigma(z)^j-z^jm_4&=(\Delta(z^{j})\cdot m_1)y- \lambda y ({}^{\sigma}\!\Delta^{\sigma}(z^{j})\cdot m_2).
\end{align*}
Thus
\begin{align*}
&\varphantom{=}\Delta(\varphi)\cdot (m_3\beta)=\sum_{i=1}^{l}a_i\sum_{j=1}^{i}\sigma(z)^{i-j}m_3\beta z^{j-1}\\
&=\sum_{i=1}^{l}\sum_{j=1}^{i}a_i\bigl(m_3z^{i-j}+x(\Delta(z^{i-j})\cdot m_1)-\lambda(\Delta(z^{i-j})\cdot m_2)x\bigr)\beta z^{j-1}\\
&=m_3\beta\varphi'+\sum_{i=2}^{l}\sum_{j=1}^{i-1}\sum_{k=1}^{i-j}a_i(xz^{i-j-k}m_1z^{j+k-2}-\lambda\sigma(z)^{i-j-k} m_2\sigma(z)^{j+k-2}x)\beta\\
&=m_3\beta\varphi'+\sum_{i=2}^{l}\sum_{j=1}^{i-1}\sum_{k=j+1}^{i}a_i(xz^{i-k}m_1z^{k-2}-\lambda\sigma(z)^{i-k} m_2\sigma(z)^{k-2}x)\beta\\
&=m_3\beta\varphi'+\sum_{i=2}^{l}\sum_{k=2}^{i}\sum_{j=1}^{k-1}a_i(xz^{i-k}m_1z^{k-2}-\lambda\sigma(z)^{i-k} m_2\sigma(z)^{k-2}x)\beta\\
&=m_3\beta\varphi'+\sum_{i=2}^{l}\sum_{k=2}^{i}(k-1)a_i(xz^{i-k}m_1z^{k-2}-\lambda\sigma(z)^{i-k} m_2\sigma(z)^{k-2}x)\beta\\
&=m_3\beta\varphi'+\sum_{i=1}^{l}\sum_{k=1}^{i}a_i(xz^{i-k}m_1(z^{k-1})'-\lambda\sigma(z^{i-k})m_2\sigma((z^{k-1})')x)
\beta\\
&=m_3\beta\varphi'+x(\Delta^D(\varphi)\cdot m_1)\beta-\lambda({}^{\sigma}\!\Delta^{\sigma D}(\varphi)\cdot m_2)\sigma(\beta)x,
\end{align*}
where $D=\rmd/\rmd z$. Similarly,
\begin{equation}\label{eq:homological-smooth-5}
{}^{\sigma}\!\Delta^{\sigma}(\varphi)\cdot (m_4\sigma(\beta))=m_4\sigma(\beta\varphi')-(\Delta^{D}(\varphi)\cdot m_1)\beta y+\lambda y({}^{\sigma}\!\Delta^{\sigma D}(\varphi)\cdot m_2)\sigma(\beta).
\end{equation}

Let $n_1=-m_3\beta$. The first component of $s^{1}(n_1,0)$ is
\begin{align*}
  &\varphantom{=}m_3-m_3\alpha yx+x(\Delta^D(\varphi)\cdot m_1)\beta-\lambda({}^{\sigma}\!\Delta^{\sigma D}(\varphi)\cdot m_2)\sigma(\beta)x\\
  &=m_3+x(\Delta^D(\varphi)\cdot m_1)\beta-(m_3\alpha y+\lambda({}^{\sigma}\!\Delta^{\sigma D}(\varphi)\cdot m_2)\sigma(\beta))x.
\end{align*}
Let
\begin{align*}
n_3&=-(\Delta^D(\varphi)\cdot m_1)\beta,\\
n_4&=-m_3\alpha y-\lambda({}^{\sigma}\!\Delta^{\sigma D}(\varphi)\cdot m_2)\sigma(\beta).
\end{align*}
Clearly, the first component of $s^{1}(n_1,0)+\partial_h^{20}(n_3,n_4)$ equals $m_3$, and the second one is equal to
\begin{equation}\label{eq:homological-smooth-6}
    (\Delta^D(\varphi)\cdot m_1)\beta y-ym_3\alpha y-\lambda y({}^{\sigma}\!\Delta^{\sigma D}(\varphi)\cdot m_2)\sigma(\beta).
\end{equation}

Next, consider the difference between $(m_1,m_2)$ and $\partial_h^{11}(n_1,0)+\partial_v^{2}(n_3,n_4)$. Notice that
\begin{equation}\label{eq:inj.der}
  z(\Delta^D(\varphi)\cdot m)-(\Delta^D(\varphi)\cdot m)z=\Delta(\varphi)\cdot m-m\varphi'
\end{equation}
for all $m\in M$. It follows that
\begin{align*}
zn_3-n_3z&=m_1\beta\varphi'-(\Delta(\varphi)\cdot m_1)\beta,\\
\sigma(z)n_4-n_4\sigma(z)&=m_3\alpha zy-\sigma(z)m_3\alpha y+\lambda m_2\sigma(\beta\varphi')-\lambda({}^{\sigma}\!\Delta^{\sigma}(\varphi)\cdot m_2)\sigma(\beta).
\end{align*}
So by \eqref{eq:homological-smooth-1}--\eqref{eq:homological-smooth-4},
\begin{align*}
  &\varphantom{=}\partial_h^{11}(n_1,0)+\partial_v^{2}(n_3,n_4)\\
  &=(ym_3\beta+m_1\beta\varphi'-(\Delta(\varphi)\cdot m_1)\beta,\, \lambda^{-1}m_3\beta y+\lambda^{-1}m_3\alpha zy\\
  &\varphantom{=}{}-\lambda^{-1}\sigma(z)m_3\alpha y+m_2\sigma(\beta\varphi')-({}^{\sigma}\!\Delta^{\sigma}(\varphi)\cdot m_2)\sigma(\beta))\\
  &=(m_1\beta\varphi'-m_4x\beta, \, -\lambda^{-1}\sigma(z)m_3\alpha y+\lambda^{-1}m_3z\alpha y-\lambda^{-1}xm_4\sigma(\beta)\\
  &\varphantom{=}{}+m_2-m_2x\alpha y)\\
  &=(m_1\beta\varphi'-m_4x\beta,\, -\lambda^{-1}xm_1\alpha y-\lambda^{-1}xm_4\sigma(\beta)+m_2)\\
  &=(m_1,m_2)-(m_1\alpha yx+m_4\sigma(\beta)x,\, \lambda^{-1}xm_1\alpha y+\lambda^{-1}xm_4\sigma(\beta))\\
  &=(m_1,m_2)-\partial_h^{11}(0, -\lambda^{-1}m_1\alpha y-\lambda^{-1}m_4\sigma(\beta)).
\end{align*}
Thus by letting $n_2=-\lambda^{-1}m_1\alpha y-\lambda^{-1}m_4\sigma(\beta)$, we have
\[
(m_1,m_2)=\partial_h^{11}(n_1,n_2)+\partial_v^{2}(n_3,n_4).
\]

In order to finish the proof, we show that the second component of $s^{1}(n_1,n_2)$ plus \eqref{eq:homological-smooth-6} is equal to $m_4$. In fact, since
\begin{align*}
  &\varphantom{=}\Delta(\varphi)\cdot(m_1\alpha y)=(\Delta(\varphi)\cdot m_1)\alpha y\\
  &=(ym_3+m_4x)\alpha y=ym_3\alpha y+m_4xy\sigma(\alpha),
\end{align*}
together with \eqref{eq:homological-smooth-5}, we have
\begin{align*}
  -\lambda\Delta(\varphi)\cdot n_2
  &=ym_3\alpha y+m_4xy\sigma(\alpha)+m_4\sigma(\beta\varphi')-(\Delta^{D}(\varphi)\cdot m_1)\beta y\\
  &\varphantom{=}{}+\lambda y({}^{\sigma}\!\Delta^{\sigma D}(\varphi)\cdot m_2)\sigma(\beta)\\
  &=m_4+ym_3\alpha y-(\Delta^{D}(\varphi)\cdot m_1)\beta y+\lambda y({}^{\sigma}\!\Delta^{\sigma D}(\varphi)\cdot m_2)\sigma(\beta)\\
  &=m_4-\eqref{eq:homological-smooth-6}.
\end{align*}

Therefore, $(m_1,m_2,m_3,m_4)=\partial^2(n_1,n_2,n_3,n_4)$, namely $\Ker\partial^3=\Ima\partial^2$, and so $H^3(\Tot\calQ^{\cdot\cdot})=0$.
\end{proof}

Next, we introduce some algebras related to $A$. Let $B_1$, $B_2$ be the subalgebras of $A$ generated by $x$ and $z$, $y$ and $z$, respectively. They are given in terms of generators and relations by
\begin{align*}
  B_1&=\kk\langle x, z\rangle/(xz-\lambda zx-\eta z),\\
  B_2&=\kk\langle y, z\rangle/(yz-\lambda^{-1} zy+\lambda^{-1}\eta z).
\end{align*}
Since both of them are Ore extensions of $\kk[z]$, by \cite{L-Wu-Wang:twisted-CY-Ore-extension}, they are twisted Calabi-Yau algebras, and their individual Nakayama automorphisms $\nu_1$, $\nu_2$ are given by $\nu_1(z)=\lambda^{-1}z-\lambda^{-1}\eta$, $\nu_1(x)=b_1 x$, $\nu_2(z)=\lambda z+\eta$, $\nu_2(y)=b_2y$ for some $b_1$, $b_2\in\kk[z]$.

On the other hand, $B_1$ and $B_2$ are endowed with the standard filtrations by $F^nB_1=\sum_{i+j\leq n}\kk z^ix^j$ and $F^nB_2=\sum_{i+j\leq n}\kk z^iy^j$, respectively. Their associated graded algebras are
\begin{align*}
  \mathrm{gr}\,B_1&=\kk\langle x, z\rangle/(xz-\lambda zx),\\
  \mathrm{gr}\,B_2&=\kk\langle y, z\rangle/(yz-\lambda^{-1} zy),
\end{align*}
which are quantum planes with their Nakayama automorphisms
\begin{alignat*}{2}
  x&\mapsto\lambda x, &\quad z&\mapsto\lambda^{-1}z,\\
  y&\mapsto \lambda^{-1}y, &\quad z&\mapsto\lambda z.
\end{alignat*}
By \cite[Proposition 1.1]{Yekutieli:rigid-dualizing-complex-universal-enveloping-algebra}, $\nu_1$, $\nu_2$ are filtered automorphisms, and the associated graded algebra automorphisms $\mathrm{gr}\,\nu_1$, $\mathrm{gr}\,\nu_2$ coincide with the Nakayama automorphisms of $\mathrm{gr}\,B_1$, $\mathrm{gr}\,B_2$, respectively. Consequently, $b_1=\lambda$, $b_2=\lambda^{-1}$.

Now extend $\sigma^{-1}$ to the automorphism $\sigma_1$ of $B_1$ by sending $x$ to $x$, and define the $\sigma_1$-derivation $\delta_1\colon B_1\to B_1$ by
\[
\delta_1|_B=0,\;\delta_1(x)=\varphi(z)-\sigma(\varphi(z)).
\]
We obtain an Ore extension $B_3=B_1[y;\sigma_1,\delta_1]$. Similarly, extend $\sigma$ to the automorphism $\sigma_2$ of $B_2$ by sending $y$ to $y$, and define the $\sigma_2$-derivation $\delta_2\colon B_2\to B_2$ by
\[
\delta_2|_B=0,\;\delta_2(y)=\sigma(\varphi(z))-\varphi(z).
\]
It is easy to check $B_3=B_2[x;\sigma_2,\delta_2]$ and that $\omega:=yx-\varphi(z)$ is a central regular element in $B_3$, and $A\cong B_3/\omega B_3$.

Applying \cite[Theorem 0.2]{L-Wu-Wang:twisted-CY-Ore-extension} again, we have

\begin{lem}\label{lem:B_3-twisted-CY}
  The algebra $B_3$ is twisted Calabi-Yau of dimension $3$ whose Nakayama automorphism $\nu_3$ is given by
\[
\nu_3(x)=\lambda x,\;\nu_3(y)=\lambda^{-1}y,\;\nu_3(z)=z.
\]
\end{lem}

In order to compute $\Ext_{A^e}^n(A,A^e)$, we need the following Rees Lemma.

\begin{lem}[{\cite[Theorem 8.34]{Rotman:homological-algebra}}]
Let $R$, $S$ be two rings, and $c\in R$ be a central element that is neither a unit nor a zero-divisor. Denote $R^*=R/c R$. If $M$ is an $R$-$S$-bimodule and $c$ is regular on $M$, then there is an isomorphism of $S$-modules
\[
\Ext^n_{R^*}(L^*, M/c M)\cong \Ext_{R}^{n+1}(L^*, M)
\]
for every $R^*$-module $L^*$ and every $n\geq 0$.
\end{lem}

\begin{prop}\label{prop:dualizing-complex}
There are isomorphisms of $A^e$-modules
\[
\Ext_{A^e}^i(A,A^e)\cong
\begin{cases}
0, & \text{if $i\neq 2$},\\
A^{\nu}, & \text{if $i=2$},
\end{cases}
\]
where the algebra automorphism $\nu$ is given by
\[
\nu(x)=\lambda x,\;\nu(y)=\lambda^{-1}y,\;\nu(z)=z.
\]
\end{prop}

\begin{proof}
First of all, since $\omega$ is a central regular element in $B_3$ and $A\cong B_3/\omega B_3$, there is a short exact sequence
\[
0\to B_3\xrightarrow{\omega} B_3\to A\to 0
\]
of $B_3^e$-modules. Applying $\Hom_{B_3^e}(-, B_3^e)$ to it, we obtain a long exact sequence
\[
\cdots \to \Ext_{B_3^e}^{i}(A, B_3^e)\to \Ext_{B_3^e}^i(B_3, B_3^e)\xrightarrow{\omega} \Ext_{B_3^e}^i(B_3, B_3^e)\to \Ext_{B_3^e}^{i+1}(A, B_3^e)\to\cdots
\]
of $B_3^e$-modules.

By Lemma \ref{lem:B_3-twisted-CY}, the above exact sequence is
\[
\cdots \to 0 \to \Ext_{B_3^e}^{3}(A, B_3^e)\to B_3^{\nu_3}\xrightarrow{\omega} B_3^{\nu_3}\to \Ext_{B_3^e}^4(A, B_3^e)\to 0 \to \cdots.
\]
Since $\omega$ is regular and $B_3^{\nu_3}/\omega B_3^{\nu_3}\cong A^\nu$, there are isomorphisms of $B_3^e$-modules
\[
\Ext_{B_3^e}^i(A, B_3^e)\cong
\begin{cases}
0, & \text{if $i\neq 4$},\\
A^\nu, & \text{if $i=4$}.
\end{cases}
\]

Apply Rees Lemma to the case: $R=S=B_3^e$, $c=\omega\otimes 1$, $L^*=A$, $M=B_3^e$. There are isomorphisms
\[
\Ext_{A\otimes B_3^{\mathrm{op}}}^i(A, A\otimes B_3^{\mathrm{op}})\cong\Ext_{B_3^e}^{i+1}(A, B_3^e)
\]
for all $i$. Apply Rees Lemma again to the case: $R=S=A\otimes B_3^{\mathrm{op}}$, $c=1\otimes\omega$, $L^*=A$, $M=A\otimes B_3^{\mathrm{op}}$. There are isomorphisms
\[
\Ext_{A^e}^{i}(A, A^e)\cong\Ext_{A\otimes B_3^{\mathrm{op}}}^{i+1}(A, A\otimes B_3^{\mathrm{op}})
\]
for all $i$. Therefore, we have
\[
\Ext_{A^e}^i(A,A^e)\cong
\begin{cases}
0, & \text{if $i\neq 2$},\\
A^{\nu}, & \text{if $i=2$}.
\end{cases}
\]
\end{proof}

Together with Lemma \ref{lem:no-multi-root}, we have

\begin{thm}\label{thm:twisted-CY}
Let  $A=\kk[z;\lambda,\eta,\varphi(z)]$, and $\nu$ be as in Proposition \ref{prop:dualizing-complex}. If $\varphi(z)$ has no multiple roots, then $A$ is homologically smooth. Moreover, $A$ is $\nu$-twisted Calabi-Yau of dimension $2$.
\end{thm}

Quantum $2$-spheres are a continuously parametrized family of $SU_q(2)$-spaces that are analogs of the classical $2$-sphere $SU(2)/SO(2)$. They were firstly constructed by Podle\'s \cite{Podles:quantum-sphere}, and later studied by many other people, see \cite{Brzezinski:quantum-homogeneous-space}, \cite{Hadfield:twisted-cyclic-homology-Podles-sphere}, \cite{Heckenberger-Kolb:Podles-sphere}, \cite{Krahmer:qh-space}, \cite{Muller-Schneider:quantum-homogeneous-space}, etc. Quantum $2$-spheres are a class of important quantum homogeneous spaces whose many properties are discovered and are applied in the realms of quantum group, noncommutative geometry, and mathematical physics.

As a $\mathbb{C}$-algebra, quantum $2$-sphere is generated by $X$, $Y$, $Z$, subject to
\begin{alignat*}{2}
& XZ= q^2ZX,&\quad& ZY= q^2YZ,\\
& YX = uv + (u-v)Z-Z^2,&\quad& XY = uv + q^2(u-v)Z-q^4Z^2,
\end{alignat*}
where $q\in\mathbb{R}^\times$ is not a root of unity, and $u$, $v\in\mathbb{R}$ with $u+v\neq 0$. It is a GWA $\mathbb{C}[Z;q^2,0,uv + (u-v)Z-Z^2]$. It is called standard if $(u,v)=(1,0)$, which turns out to be homologically smooth in \cite{Krahmer:qh-space}. The next corollary follows obviously from Theorem \ref{thm:twisted-CY} since $uv + (u-v)Z-Z^2$ has two distinct roots $u$, $-v$, and hence gives an affirmative reply to \cite[Question 2]{Krahmer:qh-space}.

\begin{cor}
The quantum $2$-spheres are all homologically smooth.
\end{cor}

\section{Noncommutative deformations}\label{sec:noncommutative-deformation}
In this section, assume that $A=\kk[z;\lambda,\eta,\varphi(z)]$ is noncommutative, namely, $(\lambda,\eta)\neq(1,0)$.

We will study deformations of $A$. A result of \cite{Richard-Solotar:isom-quantum-GWA} is that $A$ is isomorphic to a quantum GWA if $\lambda\neq 1$. The Hochschild cohomology $H^n(A,A)$ has been computed in \cite{Farinati:Hochschid-homology-GWA} (classical case) and \cite{Solotar:Hochschild-homology-GWA-quantum} (quantum case) for every $n$. So essentially, the first order deformations of $A$ have been known.

Let us focus on the formal deformations of $A$. They are controlled by $H^2(A,A)$ and $H^3(A,A)$. If one wants to construct a formal deformation, the choice of $F_i\in C^2(A,A)$ must be made carefully so that all obstructions can be passed. By Lemma \ref{lem:no-multi-root}, $H^3(A,A)=0$ if $A$ is homologically smooth. In this case, every Hochschild $2$-cocycle lifts to a formal deformation. A surprising result in this section is that even if $A$ is not homologically smooth, we can construct a formal deformation starting with a special Hochschild $2$-cocycle.

Let $M$ be an $A$-bimodule, $m\in M$. Consider the element
\[
f(m):=(\lambda mx, -ym, 0, -\lambda{}^{\sigma}\!\Delta^{\sigma}(\varphi)\cdot m)\in M^{\oplus 4}.
\]
It is easy to verify that $f(m)$ is a Per $2$-cocycle for any $m\in M$. So we obtain a map $f\colon M\to Z^2(\Tot\calQ^{\cdot\cdot})$. Let $[A,M^\nu]$ be the space of $M$ spanned by all twisted commutators $bm-m\nu(b)$ for all $b\in A$ and $m\in M$. A further computation shows $f([A,M^{\nu}])\subset B^2(\Tot\calQ^{\cdot\cdot})$, and thus $f$ induces a map $H_0(A,M^\nu)\to H^2(A,M)$. In particular, $f$ is injective when $M=A$, and so $f(A)$ is a subset of $Z^2(\Tot\calQ^{\cdot\cdot})$.

Now, $\mathcal{B}=\{z^px_q\mid p\geq 0, q\in\mathbb{Z}\}$ is a basis for $A$, define $|\!|z^px_q|\!|=p+(l+1)|q|$. This equips $A$ with a filtration $\{\Gamma A\}$ by setting
\[
\Gamma^nA=\sum_{b\in\mathcal{B},|\!|b|\!|\leq n} \kk b.
\]
A Hochschild $m$-cochain $F$ is said to preserve $\Gamma$ if $F(b_1,\ldots,b_m)\in\Gamma^{|\!|b_1|\!|+\cdots+|\!|b_m|\!|}A$ for all $b_1\ldots,b_m\in\mathcal{B}$.

\subsection{Comparisons}\label{subsec:Hochschild-cocycle}
When dealing with formal deformations, we will exchange Per cochains and Hochschild cochains frequently. For this reason, let us construct two comparisons between the bar resolution $C^{\mathrm{bar}}_\cdot$ of $A$ and $\Tot\calP_{\cdot\cdot}$
\[
\xymatrix{
A^{\otimes 2} \ar@<-.5ex>[d]_-{\theta_0} & A^{\otimes 3} \ar@<-.5ex>[d]_-{\theta_1}\ar[l]_-{b'} & A^{\otimes 4} \ar@<-.5ex>[d]_-{\theta_2}\ar[l]_-{b'} & A^{\otimes 5} \ar@<-.5ex>[d]_-{\theta_3}\ar[l]_-{b'} & \cdots\hphantom{.} \ar[l]_-{b'}\\
\calP_{00} \ar@<-.5ex>[u]_-{\theta'_0} & \calP_{01}\oplus\calP_{10} \ar@<-.5ex>[u]_-{\theta'_1}\ar[l]_-{d_1} & \calP_{11}\oplus\calP_{20} \ar@<-.5ex>[u]_-{\theta'_2}\ar[l]_-{d_2} & \calP_{21}\oplus\calP_{30} \ar@<-.5ex>[u]_-{\theta'_3}\ar[l]_-{d_3} & \cdots. \ar[l]
}
\]

In order to save space, we write $a_1|a_2|\dots|a_n$ instead of $a_1\otimes a_2\otimes\dots\otimes a_n$. Define $\theta_0=\theta'_0=\id$, and define $\theta_1$, $\theta'_1$ by
\begin{align*}
\theta_1(1|z^ix^j|1)&=\biggl(\,\sum_{k=1}^{i}z^{i-k}| z^{k-1}x^j,\,\sum_{k=1}^{j}z^ix^{j-k}| x^{k-1},\,0\biggr),\\
\theta_1(1|z^iy^j|1)&=\biggl(\,\sum_{k=1}^{i}z^{i-k}| z^{k-1}y^j,\,0,\,\sum_{k=1}^{j}z^iy^{j-k}| y^{k-1}\biggr),
\intertext{and}
\theta'_1(1|1, 0, 0)&=1|z|1,\quad \theta'_1(0, 1|1, 0)=1|x|1,\quad \theta'_1(0, 0, 1|1)=1|y|1.
\end{align*}

Recall that
\begin{align*}
d_2(1|1,0,0,0)&=(-x| 1+\lambda| x,\,\sigma(z)| 1-1| z,\,0),\\
d_2(0,1|1,0,0)&=(-y| 1+\lambda^{-1}| y,\,0, \,\lambda^{-1}z| 1-\lambda^{-1}|\sigma(z)),\\
d_2(0,0,1|1,0)&=(-\Delta(\varphi),\, y| 1, \,1| x),\\
d_2(0,0,0,1|1)&=(-\lambda{}^{\sigma}\!\Delta^{\sigma}(\varphi),\,1| y, \,x| 1).
\end{align*}
By the fact $\Delta(\sigma^q(z)^i)=\lambda^q\sum_{k=1}^{i}\sigma^q(z)^{i-k}| \sigma^q(z)^{k-1}$, we have
\begin{align*}
&\varphantom{=}\theta_1b'(1|z^px^q|z^ix^j|1)\\
&=\underbrace{\theta_1(z^px^q|z^ix^j|1)}_{\mathrm{Part}\,1}-\underbrace{\theta_1(1|z^p\sigma^q(z)^ix^{q+j}|1)}
_{\mathrm{Part}\,2}+\underbrace{\theta_1(1|z^px^q|z^ix^j)}_{\mathrm{Part}\,3}
\end{align*}
and the three parts are
\begin{align*}
\mathrm{Part}\,1&=\biggl(\,\sum_{k=1}^{i}z^px^qz^{i-k}| z^{k-1}x^j,\,\sum_{k=1}^{j}z^px^qz^ix^{j-k}| x^{k-1},\,0\biggr),\\
\mathrm{Part}\,2&=\biggl(\,\sum_{k=1}^{p}z^{p-k}| z^{k-1}\sigma^q(z)^ix^{q+j}+\sum_{k=1}^{i}\lambda^qz^p\sigma^q(z)^{i-k}|\sigma^q(z)^{k-1}x^{q+j}, \\
&\phantom{=\biggl(}\,\sum_{k=1}^{q+j}z^p\sigma(z)^ix^{q+j-k}| x^{k-1},\,0\biggr),\\
\mathrm{Part}\,3&=\biggl(\,\sum_{k=1}^{p}z^{p-k}|z^{k-1}x^qz^ix^j,\, \sum_{k=1}^{q}z^px^{q-k}|x^{k-1}z^ix^{j},\,0\biggr),
\end{align*}
respectively. A direct computation shows that the sum is equal to
\[
d_2\biggl(-\sum_{k=1}^{i}\sum_{s=1}^{q}z^p\sigma^q(z)^{i-k}x^{q-s}|(\lambda x)^{s-1} z^{k-1}x^{j},0,0,0\biggr).
\]
Denote $\Delta^{\nu}(x^q)=\sum_{s=1}^{q}x^{q-s}|(\lambda x)^{s-1}$ and thus define
\[
\theta_2(1|z^px^q|z^ix^j|1)=\bigl(-z^p({}^{\sigma^q}\!\Delta(z^{i})\cdot\Delta^{\nu}(x^q) )x^{j},\,0,\,0,\,0\bigr).
\]
Similarly, we define
\begin{align*}
\theta_2(1|z^px|z^iy^j|1)&=\bigl( -z^p\,{}^{\sigma}\!\Delta(z^{i})y^j,\,0,\,0,\,z^p\sigma(z^i)| y^{j-1}\bigr),\\
\theta_2(1|z^py|z^ix^j|1)&=\bigl(0,\,-z^p\,{}^{\sigma^{-1}}\!\!\Delta(z^{i})x^j,\,z^p\sigma^{-1}(z^i)| x^{j-1},\,0\bigr),\\
\theta_2(1|z^py^q|z^iy^j|1)&=\bigl(0,\,-z^p({}^{\sigma^{-q}}\!\!\Delta(z^{i})\cdot \Delta^{\nu}(y^q))y^j,\,0,\,0\bigr),\\
\theta_2(1|z^p|z^ix^j|1)&=\theta_2(1|z^p|z^iy^j|1)=0
\end{align*}
where $\Delta^{\nu}(y^q)=\sum_{s=1}^{q}y^{q-s}|(\lambda^{-1} y)^{s-1}$.

The explicit forms of $\theta_2(1|z^px^q|z^iy^j|1)$ and $\theta_2(1|z^py^q|z^ix^j|1)$ for $q\geq 2$ can be obtained by a further (but boring) computation. The following subsection concerns Hochschild $2$-cochains $F=f\circ\theta_2$ for some Per $2$-cochains. By virtue of the following lemma, we do not have to describe $F(u,v)$ for every $(u,v)\in A\times A$ if $F$ satisfies some conditions. In particular, the computation of $\theta_2(1|z^px^q|z^iy^j|1)$, $\theta_2(1|z^py^q|z^ix^j|1)$ for $q\geq 2$ is actually avoidable.

\begin{lem}\label{lem:determine-F_n}
  Let $A=\kk[z;\lambda,\eta,\varphi(z)]$, and $F\colon A\times A\to A$ be a Hochschild $2$-cochain. Suppose for any $u$, $v\in A$,
  \begin{enumerate}[(a)]
  \item $F(u,1)=F(1,v)=0$,
  \item $F(zu,v)=zF(u,v)$, and $F(x^q,x^j)=F(y^q,y^j)=0$ for all $q$, $j$.
  \end{enumerate}
  Then $F$ is uniquely determined by the following datum:
  \[
    \mathsf{b}F,\,F(x,z),\,F(x,y),\,F(y,z),\,F(y,x).
  \]
  If further, $\mathsf{b}F$ preserves $\Gamma$, $F(x,z)$, $F(y,z)\in\Gamma^{l+2}A$ and $F(x,y)$, $F(y,x)\in\Gamma^{2l+2}A$, then $F$ preserves $\Gamma$.
\end{lem}

\begin{proof}
  By the definition of $\mathsf{b}F$, we have
  \begin{align}
    \mathsf{b}F(x,z,z^{i-1})&=xF(z, z^{i-1})-F(xz,z^{i-1})+F(x,z^i)-F(x,z)z^{i-1}\notag\\
    &=F(x,z^i)-\sigma(z)F(x,z^{i-1})-F(x,z)z^{i-1},\label{eq:lem1}\\
    \mathsf{b}F(x,y,y^{j-1})&=xF(y, y^{j-1})-F(xy,y^{j-1})+F(x,y^j)-F(x,y)y^{j-1}\notag\\
    &=F(x,y^j)-F(x,y)y^{j-1},\label{eq:lem2}\\
    \mathsf{b}F(x,z^i,y^j)&=xF(z^i, y^j)-F(xz^i,y^j)+F(x,z^iy^j)-F(x,z^i)y^j\notag\\
    &=F(x,z^i y^j)-\sigma(z)^i F(x,y^j)-F(x,z^i)y^j\label{eq:lem3}\\
    \mathsf{b}F(x,z^i,x^j)&=xF(z^i, x^j)-F(xz^i,x^j)+F(x,z^ix^j)-F(x,z^i)x^j\notag\\
    &=F(x,z^i x^j)-F(x,z^i)x^j.\label{eq:lem4}
  \end{align}
  Since \eqref{eq:lem1} is a recursive formula for $F(x,z^i)$, it determines $F(x,z^i)$. The value of $F(x,y^j)$ is immediately from \eqref{eq:lem2}. Then by \eqref{eq:lem3}, \eqref{eq:lem4}, $F(x,z^ix^j)$, $F(x,z^iy^j)$ are also determined.

  For $F(x^q, z^ix^j)$, we just consider $\mathsf{b}F(x^{q-1},x,z^ix^j)$. For $F(x^q, z^iy^j)$, if $j\geq 1$, then
  \[
  F(x^q, z^iy^j)=F(x^{q-1},\sigma(z)^i\sigma(\varphi)y^{j-1})+x^{q-1}F(x, z^iy^j)-\mathsf{b}F(x^{q-1},x,z^iy^j).
  \]
  So it reduces to determine $F(x^{q-1},z^iy^{j-1})$. This can be achieved by induction on $q$.

  Since $F(zu,v)=zF(u,v)$, all the values of $F(z^px^q, z^ix^j)$, $F(z^px^q, z^iy^j)$ are uniquely determined, and similar for $F(z^py^q, z^ix^j)$, $F(z^py^q, z^iy^j)$.

  The last claim is contained in the above argument.
\end{proof}

\begin{remark}\label{rk:determing-Fn}
If $\mathsf{b}F=0$, then by the above proof, the first equation of $(b)$ holds for all $u$, $v\in A$ if it holds for $u\in\{x,y,z\}$, $v\in A$.
\end{remark}

The morphisms $\theta'_2$, $\theta'_3$ are listed as follows.
\begin{align*}
\theta'_2(1|1, 0, 0, 0)&=\lambda|z|x|1-1|x|z|1,\\
\theta'_2(0, 1|1, 0, 0)&=\lambda^{-1}|z|y|1-1|y|z|1,\\
\theta'_2(0, 0, 1|1, 0)&=1|y|x|1+1|1|1|\varphi(z)-\sum_{i=1}^{l}\sum_{j=1}^{i}a_i|z^{i-j}|z|z^{j-1},\\
\theta'_2(0, 0, 0, 1|1)&=1|x|y|1+1|1|1|\sigma(\varphi(z))-\sum_{i=1}^{l}\sum_{j=1}^{i}a_i|\sigma(z)^{i-j}|\lambda z|\sigma(z)^{j-1},\\
\theta'_3(1|1, 0, 0, 0)&=1|z|y|x|1-\lambda|y|z|x|1+1|y|x|z|1+1|z|1|1|\varphi(z)\\
&\varphantom{=}{}+1|1|1|z|\varphi(z)-\sum_{i=1}^{l}\sum_{j=1}^{i}a_i|z|z^{i-j}|z|z^{j-1},\\
\theta'_3(0, 1|1, 0, 0)&=1|z|x|y|1-\lambda^{-1}|x|z|y|1+1|x|y|z|1+1|z|1|1|\sigma(\varphi(z))\\
&\varphantom{=}{}+1|1|1|z|\sigma(\varphi(z))-\sum_{i=1}^{l}\sum_{j=1}^{i}a_i|z|\sigma(z)^{i-j}|\lambda z|\sigma(z)^{j-1},\\
\theta'_3(0, 0, 1|1, 0)&=1|x|y|x|1-\sum_{i=1}^{l}\sum_{j=1}^{i}\bigl(a_i|x|z^{i-j}|z|z^{j-1}-a_i|\sigma(z)^{i-j}|x|z|z^{j-1}\\
&\varphantom{=}{}+a_i|\sigma(z)^{i-j}|\lambda z|x|z^{j-1}\bigr)+1|x|1|1|\varphi(z)+1|1|1|x|\varphi(z),\\
\theta'_3(0, 0, 0, 1|1)&=1|y|x|y|1-\sum_{i=1}^{l}\sum_{j=1}^{i}\bigl(a_i|y|\sigma(z)^{i-j}|\lambda z|\sigma(z)^{j-1}\\
&\varphantom{=}{}-a_i|z^{i-j}|y|\lambda z|\sigma(z)^{j-1}+a_i|z^{i-j}|z|y|\sigma(z)^{j-1}\bigr).
\end{align*}

Higher $\theta_i$, $\theta'_i$ can be found inductively; the above is as much as we will need in the following.

\subsection{Deformations of GWAs (I)}\label{subsec:deformation}
From now on, $\kk$ is of characteristic zero. Denote $\bar{\varphi}(z)=\sigma(\varphi(z))$. Let $A=\kk[z;\lambda,0,\varphi(z)]$ and we will construct a formal deformation starting with the Per $2$-cocycle $f(z)$.

Let $F_1$ be the Hochschild $2$-cocycle corresponding to $f(z)$. By the map $\theta_2$ defined in the previous subsection, we have
\begin{align*}
F_1(z^px^q,z^ix^j)&=-\lambda z^p(\Delta^{\nu}(x^q)\cdot z)x(z^i)'x^j,\\
F_1(z^px,z^iy^j)&=-\lambda z^{p+1}x(z^i)'y^j-z^{p+1}\bar{\varphi}'(z)(\lambda z)^iy^{j-1},\\
F_1(z^py,z^ix^j)&=z^pyz(z^i)'x^j,\\
F_1(z^py^q,z^iy^j)&=z^py(\Delta^{\nu}(y^q)\cdot z)(z^i)'y^j.
\end{align*}
By Remark \ref{rk:determing-Fn}, $F_1$ satisfies the conditions (a), (b) in Lemma \ref{lem:determine-F_n} and so $F_1$ is uniquely determined by these equations. Also, $F_1$ preserves $\Gamma$. The fact $F_1$ satisfies the condition (a) is equivalent to that the identity of the first order deformation is $1_A$. It is reasonable to require that the undetermined maps $F_2, F_3, \ldots$ also satisfy the condition (a).

Identify $F_1\bullet F_1$ with a homomorphism in $\Hom_{A^e}(A^{\otimes 5}, A)$. By the definition of $\theta'_3$, the map $(F_1\bullet F_1)\circ\theta'_3$ corresponds to the Per $3$-cocycle
\[
\biggl(zyx,\,zxy,\,-\frac{1}{2}z^2\bar{\varphi}''(z)x,\,yz\bar{\varphi}'(z)+\frac{1}{2}yz^2
\bar{\varphi}''(z)\biggr).
\]
It is a Per $3$-coboundary, and one of its preimages under $\partial^2$ is
\[
\biggl(0,\, -yz,\, 0, \,\frac{1}{2}z^2\bar{\varphi}''(z)\biggr)
\]
which corresponds to an $A^e$-module homomorphism $f_2\colon\calP_{11}\oplus \calP_{20}\to A$. Let us give a $\kk$-bilinear map $F_2\colon A\times A\to A$ lifting $f_2$, that is, $f_2=F_2\circ\theta'_2$. Since $F_2$ satisfies the condition (a), we obtain a system of equations
\begin{equation}\label{eq:system-equations}
\left\{
\begin{alignedat}[r]{2}
&\vphantom{\bigg(}\lambda F_2(z,x)-F_2(x,z)&&=0\\
&\lambda^{-1}F_2(z,y)-F_2(y,z)&&=-yz\\
&F_2(y,x)-\sum_{i=1}^{l}\sum_{j=1}^{i}a_iF_2(z^{i-j},z)z^{j-1}&&=0\\
&F_2(x,y)-\sum_{i=1}^{l}\sum_{j=1}^{i}a_i\lambda^iF_2(z^{i-j},z)z^{j-1}&&=\frac{1}{2}z^2\bar{\varphi}''(z).
\end{alignedat}
\right.
\end{equation}
Among the solutions, the only one also satisfying the condition (b) in Lemma \ref{lem:determine-F_n} is given by
\begin{gather*}
F_2(x,z)=0,\,F_2(y,z)=yz,\\
F_2(x,y)=\frac{1}{2}z^2\bar{\varphi}''(z),\,F_2(y,x)=0.
\end{gather*}
These equations, together with $F_1\bullet F_1=\mathsf{b}F_2$, determine the map $F_2$ by Lemma \ref{lem:determine-F_n}. What is more, the fact $F_1$ preserves $\Gamma$ implies that $\mathsf{b}F_2$, and hence $F_2$, both preserve $\Gamma$. According to the proof of Lemma \ref{lem:determine-F_n}, we have
\begin{gather*}
F_2(x,yz)=z^2\bar{\varphi}'(z)+\frac{1}{2}z^3\bar{\varphi}''(z),\,F_2(y,xz)=zyx,\\
F_2(x,h(z))=\frac{1}{2}xz^2h''(z),\,F_2(y,h(z))=\frac{1}{2}yz(zh(z))''.
\end{gather*}

Next consider the map $F_1\bullet F_2+F_2\bullet F_1$ and continue the procedure successively.

\begin{thm}\label{thm:formal-deformation-quantum}
Let $A=\kk[z;\lambda,0,\varphi(z)]$ and $F_1$ correspond to $f(z)$. There exist uniquely a family of $\kk$-bilinear maps $F_n\colon A\times A\to A$, $n\geq 2$ integrating $F_1$ that satisfy the conditions (a), (b) in Lemma \ref{lem:determine-F_n} and are determined by
\begin{gather*}
F_n(x,z)=0,\,F_n(y,z)=yz,\\
F_n(x,y)=\frac{(-1)^n}{n!}z^n\bar{\varphi}^{(n)}(z),\,F_n(y,x)=0.
\end{gather*}
Moreover, they preserve $\Gamma$.
\end{thm}

\begin{proof}
The uniqueness of $F_n$ follows from Lemma \ref{lem:determine-F_n} whenever $F_n$ exists. Let us prove by induction the existence and that the equations
\begin{gather}
F_n(x,yz)=\frac{(-1)^n}{(n-1)!}z^n\bar{\varphi}^{(n-1)}(z)+\frac{(-1)^{n}}{n!}z^{n+1}\bar{\varphi}^{(n)}(z), \label{eq:eq1-quantum}\\
F_n(y,xz)=zyx,\label{eq:eq2-quantum}\\
F_n(x,h(z))=\frac{(-1)^n}{n!}xz^nh^{(n)}(z),\label{eq:eq3-quantum}\\
F_n(y,h(z))=\frac{1}{n!}yz(z^{n-1}h(z))^{(n)}\label{eq:eq4-quantum}
\end{gather}
are fulfilled.

Assume $n\geq 3$ and $F_2, F_3, \ldots, F_{n-1}$ exist. Suppose that $\theta'_3$ followed by
\[
F_1\bullet F_{n-1}+F_2\bullet F_{n-2}+\dots+F_{n-1}\bullet F_1,
\]
corresponds the Per $3$-cocycle $(S_1,S_2,S_3,S_4)$. We have
\begin{align*}
  S_1&=\sum_{i=1}^{n-1}\bigl(-\lambda F_i\bullet F_{n-i}(y,z,x)+ F_i\bullet F_{n-i}(y,x,z)\bigr)=zyx,\\
  S_2&=\sum_{i=1}^{n-1}\bigl(-\lambda^{-1}F_i\bullet F_{n-i}(x,z,y)+F_i\bullet F_{n-i}(x,y,z)\bigr)=zxy,\\
  S_3&=\sum_{i=1}^{n-1}\biggl(F_i\bullet F_{n-i}(x,y,x)-\sum_{j=1}^l\sum_{k=1}^ja_jF_i\bullet F_{n-i}(x,z^{j-k},z)z^{k-1}\biggr)\\
  &=\sum_{i=1}^{n-1}\biggl(F_i(F_{n-i}(x,y),x)-\sum_{j=1}^l\sum_{k=1}^ja_jF_i(F_{n-i}(x,z^{j-k}),z)z^{k-1}\biggr)\\
  &=-\sum_{i=1}^{n-1}\sum_{j=1}^l\sum_{k=1}^ja_jF_i\biggl(\frac{(-1)^{n-i}}{(n-i)!} xz^{n-i}(z^{j-k})^{(n-i)},z\biggr)z^{k-1}\\
  &=-\sum_{j=1}^l\sum_{k=1}^ja_jF_1\biggl((-1)^{n-1}\binom{j-k}{n-1} xz^{j-k},z\biggr)z^{k-1}\\
  &=\sum_{j=1}^l\sum_{k=1}^ja_j(-1)^{n-1}\binom{j-k}{n-1} (\lambda z)^{j}x\\
  &=\sum_{j=1}^la_j(-1)^{n-1}\binom{j}{n} (\lambda z)^{j}x\\
  &=\frac{(-1)^{n-1}}{n!}z^n\bar{\varphi}^{(n)}(z)x.
\end{align*}
For $S_4$, the computation is more complicated. We have
\begin{align*}
  S_4&=\sum_{i=1}^{n-1}\biggl(F_i\bullet F_{n-i}(y,x,y)-\sum_{j=1}^l\sum_{k=1}^ja_jF_i\bullet F_{n-i}(y,(\lambda z)^{j-k},\lambda z)(\lambda z)^{k-1}\biggr)\\
  &=\sum_{i=1}^{n-1}\biggl(-F_i(y,F_{n-i}(x,y)) -\sum_{j=1}^l\sum_{k=1}^ja_j\lambda^jF_i(F_{n-i}(y,z^{j-k}),z)z^{k-1}\biggr)\\
  &=\sum_{i=1}^{n-1}F_i\biggl(-y,\frac{(-1)^{n-i}}{(n-i)!}z^{n-i}\bar{\varphi}^{(n-i)}(z)\biggr)\\
  &\varphantom{=}{}-\sum_{i=1}^{n-1}\sum_{j=1}^l\sum_{k=1}^ja_j\lambda^j F_i\biggl(\frac{1}{(n-i)!}yz(z^{n-i+j-k-1})^{(n-i)},z\biggr)z^{k-1}\\
  &=-\sum_{i=1}^{n-1}\frac{1}{i!}yz\biggl(\frac{(-1)^{n-i}}{(n-i)!}z^{n-1}\bar{\varphi}^{(n-i)}(z)\biggr)^{(i)}\\
  &\varphantom{=}{}-\sum_{i=1}^{n-1}\sum_{j=1}^l\sum_{k=1}^ja_j\lambda^k\binom{n-i+j-k-1}{n-i}z^{j-k} F_i(y,z)z^{k-1}\\
  &=-\sum_{i=1}^{n-1}\frac{(-1)^{n-i}}{i!(n-i)!}yz(z^{n-1}\bar{\varphi}^{(n-i)}(z))^{(i)}\\
  &\varphantom{=}{}-\sum_{i=1}^{n-1}\sum_{j=1}^l\sum_{k=1}^ja_j\lambda^j\binom{n-i+j-k-1}{n-i}yz^j\\
  &=\frac{(-1)^n}{n!}yz^n\bar{\varphi}^{(n)}(z)-S_4'-S_4'',
\end{align*}
where
\begin{align*}
  S_4'&=\sum_{i=0}^{n-1}\frac{(-1)^{n-i}}{i!(n-i)!}yz(z^{n-1}\bar{\varphi}^{(n-i)}(z))^{(i)}\\
  &=\sum_{i=0}^{n-1}\sum_{k=0}^i\frac{(-1)^{n-i}}{i!(n-i)!}\binom{i}{k}yz(z^{n-1})^{(k)}\bar{\varphi}^{(n-k)}(z)\\
  &=\sum_{k=0}^{n-1}\sum_{i=k}^{n-1}\frac{(-1)^{n-i}}{i!(n-i)!}\binom{i}{k}yz(z^{n-1})^{(k)}\bar{\varphi}^{(n-k)}(z)\\
  &=\sum_{k=0}^{n-1}\sum_{i=0}^{n-k-1}\frac{(-1)^{n-i-k}}{(i+k)!(n-i-k)!}\binom{i+k}{k} yz(z^{n-1})^{(k)}\bar{\varphi}^{(n-k)}(z)\\
  &=\sum_{k=0}^{n-1}\sum_{i=0}^{n-k-1}(-1)^{i}\binom{n-k}{i}\frac{(-1)^{n-k}}{(n-k)!} \binom{n-1}{k}yz^{n-k}\bar{\varphi}^{(n-k)}(z)\\
  &=-\sum_{k=0}^{n-1}\frac{1}{(n-k)!}\binom{n-1}{k}yz^{n-k}\bar{\varphi}^{(n-k)}(z)\\
  &=-\frac{1}{n!}yz\sum_{k=0}^{n-1}\binom{n}{k}(z^{n-1})^{(k)}\bar{\varphi}^{(n-k)}(z)\\
  &=-\frac{1}{n!}yz(z^{n-1}\bar{\varphi}(z))^{(n)},
\end{align*}
and
\begin{align*}
  S_4''&=\sum_{i=1}^{n-1}\sum_{j=1}^l\sum_{k=1}^ja_j\lambda^j\binom{n-i+j-k-1}{n-i}yz^j\\
  &=\sum_{i=1}^{n-1}\sum_{j=1}^la_j\lambda^j\binom{n-i+j-1}{n-i+1}yz^j\\
  &=\sum_{i=1}^{n-1}\sum_{j=1}^la_j\lambda^j\binom{n-i+j-1}{j-2}yz^j\\
  &=\sum_{j=1}^la_j\lambda^j\biggl(\binom{n+j-1}{j-1}-j\biggr)yz^j\\
  &=\sum_{j=1}^la_j\lambda^j\binom{n+j-1}{n}yz^j-\sum_{j=1}^la_j\lambda^jjyz^j\\
  &=\frac{1}{n!}yz(z^{n-1}\bar{\varphi}(z))^{(n)}-yz\bar{\varphi}'(z).
\end{align*}
Thus
\[
S_4=yz\bar{\varphi}'(z)+\frac{(-1)^n}{n!}yz^n\bar{\varphi}^{(n)}(z).
\]

Choose a preimage of $(S_1,S_2,S_3,S_4)$ to be
\[
\biggl(0,\, -yz,\,0,\, \frac{(-1)^n}{n!}z^n\bar{\varphi}^{(n)}(z)\biggr),
\]
and then establish a system of equations similar to \eqref{eq:system-equations}, whose solution satisfying the conditions (a), (b) in Lemma \ref{lem:determine-F_n} is
\begin{gather*}
F_n(x,z)=0,\,F_n(y,z)=yz,\\
F_n(x,y)=\frac{(-1)^n}{n!}z^n\bar{\varphi}^{(n)}(z),\,F_n(y,x)=0.
\end{gather*}
By \eqref{eq:obstruction} and the induction hypotheses, $\mathsf{b}F_n$ preserves $\Gamma$. Henceforth we deduce from the above equations that $F_n$ also preserves $\Gamma$.

Finally, let us verify the equalities \eqref{eq:eq1-quantum}--\eqref{eq:eq4-quantum}. On one hand,
\begin{align*}
  \sum_{i=1}^{n-1}F_{i}\bullet F_{n-i}(x,y,z)&=\sum_{i=1}^{n-1}F_{i}( F_{n-i}(x,y),z)-\sum_{i=1}^{n-1}F_{i}(x, F_{n-i}(y,z))\\
  &=\sum_{i=1}^{n-1}F_{i}\biggl(\frac{(-1)^{n-i}}{(n-i)!}z^{n-i}\bar{\varphi}^{(n-i)}(z),z\biggr) -\sum_{i=1}^{n-1}F_{i}(x, yz)\\
  &=\sum_{i=1}^{n-1}\frac{(-1)^{i-1}}{(i-1)!}z^i\bar{\varphi}^{(i-1)}(z) -\sum_{i=1}^{n-1}\frac{(-1)^{i}}{i!}z^{i+1}\bar{\varphi}^{(i)}(z)\\
  &=z\bar{\varphi}(z) -\frac{(-1)^{n-1}}{(n-1)!}z^{n}\bar{\varphi}^{(n-1)}(z)\\
  &=zxy +\frac{(-1)^{n}}{(n-1)!}z^{n}\bar{\varphi}^{(n-1)}(z).
\end{align*}
On the other hand, $\mathsf{b}F_n=\sum_{i=1}^{n-1}F_{i}\bullet F_{n-i}$, so
\begin{align*}
 \sum_{i=1}^{n-1}F_{i}\bullet F_{n-i}(x,y,z)&=x F_n(y,z)-F_n(xy,z)+F_n(x,yz)-F_n(x,y)z\\
 &=xyz+F_n(x,yz)-\frac{(-1)^n}{n!}z^n\bar{\varphi}^{(n)}(z)z.
\end{align*}
It follows that
\[
F_n(x,yz)=\frac{(-1)^{n}}{(n-1)!}z^{n}\bar{\varphi}^{(n-1)}(z)+\frac{(-1)^n}{n!}z^{n+1}\bar{\varphi}^{(n)}(z),
\]
namely, \eqref{eq:eq1-quantum} holds.

For \eqref{eq:eq3-quantum}, by computing $\sum_{i=1}^{n-1}F_{i}\bullet F_{n-i}(x,z^{r-1},z)$ and $\mathsf{b}F_n(x,z^{r-1},z)$, one has
\begin{align*}
&\begin{aligned}
  \sum_{i=1}^{n-1}F_{i}\bullet F_{n-i}(x,z^{r-1},z)&=\sum_{i=1}^{n-1}F_{i}( F_{n-i}(x,z^{r-1}),z)\\
  &=\sum_{i=1}^{n-1}F_{i}\biggl(\frac{(-1)^{n-i}}{(n-i)!}xz^{n-i}(z^{r-1})^{(n-i)},z\biggr)\\
  &=-\frac{(-1)^{n-1}}{(n-1)!}xz^{n}(z^{r-1})^{(n-1)}\\
  &=(-1)^{n}\binom{r-1}{n-1}xz^r,
\end{aligned}\\
&\begin{aligned}
  \mathsf{b}F_n(x,z^{r-1},z)&=x F_n(z^{r-1},z)-F_n(xz^{r-1},z)+F_n(x,z^{r})-F_n(x,z^{r-1})z\\
  &=F_n(x,z^{r})-F_n(x,z^{r-1})z.
\end{aligned}
\end{align*}
Thus
\begin{align*}
 F_n(x,z^{r})&=F_n(x,z^{r-1})z+(-1)^{n}\binom{r-1}{n-1}xz^r\\
 &=F_n(x,z^{r-2})z^2+(-1)^{n}\binom{r-2}{n-1}xz^r+(-1)^{n}\binom{r-1}{n-1}xz^r\\
 &\varphantom{=}\vdots\\
 &=(-1)^{n}\biggl(\binom{1}{n-1}+\binom{2}{n-1}+\dots+\binom{r-1}{n-1}\biggr)xz^r\\
 &=(-1)^{n}\binom{r}{n}xz^r\\
 &=\frac{(-1)^{n}}{n!}xz^n(z^r)^{(n)}
\end{align*}
and so
\[
F_n(x,h(z))=\frac{(-1)^{n}}{n!}xz^nh^{(n)}(z).
\]

Similarly, \eqref{eq:eq2-quantum} and \eqref{eq:eq4-quantum} can be proved.
\end{proof}

Henceforth, we obtain a formal deformation $(A[[\tau]],*)$ of $A$. The multiplication on generators is given by
\begin{gather*}
z*x=zx,\,z*y=zy,\,x*x=x^2,\,y*y=y^2,\\
z*z=z^2,\, x*z=xz-\tau xz,\, y*z=\sum_{n=0}^{\infty}\tau^nyz,\\
x*y=xy+\sum_{n=1}^{\infty}\frac{(-1)^n}{n!}\tau^nz^n\bar{\varphi}^{(n)}(z),\, y*x=yx=\varphi(z).
\end{gather*}
Since every $F_n$ preserves $\Gamma$ and $\dim_{\kk} \Gamma^mA<\infty$ for all $m$, $*$ is locally finite. This fact gives rise to a subalgebra $\tilde{A}$ of $A[[t]]$,  in terms of generators and relations, $\tilde{A}=\kk[[\tau]]\langle x, y, z\rangle/(f_1,f_2,f_3,f_4)$ where
\begin{align*}
f_1&:=xz+\tau \lambda zx-\lambda zx=xz-(1-\tau)\lambda zx,\\
f_2&:=yz-\sum_{n=1}^{\infty}\tau^n\lambda^{-1}zy-\lambda^{-1}zy=yz-(1-\tau)^{-1}\lambda^{-1}zy,\\
f_3&:=xy-\sum_{n=1}^{\infty}\frac{(-1)^n}{n!}\tau^nz^n\bar{\varphi}^{(n)}(z)-\bar{\varphi}(z)=xy-\bar{\varphi}(z-\tau z),\\
f_4&:=yx-\varphi(z).
\end{align*}
The locally finite deformation $\tilde{A}_{\tau}$ is a quantum GWA $\kk(\!(\tau)\!)[z;(1-\tau)\lambda,0,\varphi(z)]$.

Now let us consider the classical case $A=\kk[z;1,\eta,\varphi(z)]$. Let $F_1$ be the Hochschild $2$-cocycle corresponding to $f(1)$. Thus
\begin{gather*}
F_1(x,z)=-x,\,F_1(y,z)=y,\, F_1(z,x)=0,\\
F_1(x,y)=-\bar{\varphi}'(z),\,F_1(y,x)=0,\, F_1(z,y)=0.
\end{gather*}
By a similar (and easier) argument with Theorem \ref{thm:formal-deformation-quantum}, we have

\begin{thm}\label{thm:formal-deformation-classical}
Let $A=\kk[z;1,\eta,\varphi(z)]$ and $F_1$ correspond to $f(1)$. There exist uniquely a family of $\kk$-bilinear maps $F_n\colon A\times A\to A$, $n\geq 2$ integrating $F_1$ that satisfy the conditions (a), (b) in Lemma \ref{lem:determine-F_n} and are determined by
\begin{gather*}
F_n(x,z)=0,\,F_n(y,z)=0,\\
F_n(x,y)=\frac{(-1)^n}{n!}\bar{\varphi}^{(n)}(z),\,F_n(y,x)=0.
\end{gather*}
Moreover, they preserve $\Gamma$.
\end{thm}

Like the quantum case, we obtain a formal deformation $(A[[\tau]],*)$, as well as a locally finite deformation $\tilde{A}_{\tau}=\kk(\!(\tau)\!)\langle x, y, z\rangle/(f_1,f_2,f_3,f_4)$ where
\begin{align*}
f_1&:=xz+\tau x- (z+\eta)x=xz-(z+\eta-\tau) x,\\
f_2&:=yz-\tau y-(z-\eta)y=yz-(z-\eta+\tau)y,\\
f_3&:=xy-\sum_{n=1}^{\infty}\frac{(-1)^n}{n!}\tau^n\bar{\varphi}^{(n)}(z)-\bar{\varphi}(z)=xy-\bar{\varphi}(z-\tau),\\
f_4&:=yx-\varphi(z).
\end{align*}
Obviously, $\tilde{A}_{\tau}$ is a classical GWA $\kk(\!(\tau)\!)[z;1,\eta-\tau,\varphi(z)]$.

Summarizing both cases,
\begin{enumerate}
\item $\tilde{A}_\tau$ is a noncommutative GWA over the field $\kk(\!(\tau)\!)$,
\item $\tilde{A}_\tau$ is homologically smooth if and only if $A$ is also,
\item $\tilde{A}_\tau$ is quantum (resp.\ classical) if $A$ is quantum (resp.\ classical).
\end{enumerate}

\subsection{Deformations of GWAs (II)}\label{subsec:deformation-comm}
In the foregoing subsection, we studied deformations of noncommutative GWAs. We ask the opposite question: Can we obtain a noncommutative GWA by deforming a commutative algebra? In this subsection we will give a positive answer under an assumption of the field $\kk$.

Let us first give a brief review of T.J.~Hodges's result \cite{Hodge:Kleinian-singularities}: When $\deg\varphi(z)\geq 2$, $A=\kk[z;1,1,\varphi(z)]$ can be viewed as a deformation of type-$A$ Kleinian singularity.

Recall $\varphi(z)=a_lz^l+\cdots a_1z+a_0$, $a_l\neq 0$. Equip $A$ with a filtration by
\[
F^nA=\sum_{2i+lj\leq n}\kk z^ix^j+\sum_{2i+lj\leq n}\kk z^iy^j
\]
and the associated graded algebra is
\[
\mathrm{gr}A=\kk[x,y,z]/(xy-a_lz^l).
\]
Since $l\geq 2$, after a suitable linear transformation, $\mathrm{gr}A$ is isomorphic to the coordinate algebra of the type-$A_{l-1}$ Kleinian singularity $x^2+y^2+z^l$.

\begin{remark}
We mention in passing that bimodule projective resolutions of the coordinate algebras of type-$A$ Kleinian singularities, or more generally, a class of algebras $S=\kk[x,y,z]/(x^m+y^n+z^l)$, $m$, $n$, $l\geq 2$, can be constructed using homotopy double complexes. Let $R=\kk[x,y]$ and $\rmd s=s\otimes 1-1\otimes s$. By the exact complexes
\begin{gather}
\cdots \to S\otimes_R S\xrightarrow{\Delta(z^l)} S\otimes_R S\xrightarrow{\rmd z} S\otimes_R S\xrightarrow{\Delta(z^l)} S\otimes_R S\xrightarrow{\rmd z} S\otimes_R S\to S\to 0,\label{eq:complex1}\\
0\to S\otimes S\xrightarrow{\binom{\rmd y}{-\rmd x}} (S\otimes S)^2\xrightarrow{\scriptscriptstyle(\rmd x\,\, \rmd y)} S\otimes S\to S\otimes_R S\to 0,\label{eq:complex2}
\end{gather}
we construct a homotopy double complex
\[
\xymatrix@R=8mm@C=8mm{
\cdots \ar[r] & S\otimes S \ar[d]\ar[r] & S\otimes S \ar[d]\ar[r] & S\otimes S \ar[d]\ar[r] & S\otimes S \ar[d]\ar[r] & S\otimes S \ar[d] \\
\cdots \ar[r] & (S\otimes S)^2 \ar[d]\ar[r]\ar[rru]|!{[ru];[r]}\hole & (S\otimes S)^2 \ar[d]\ar[r]\ar[rru]|!{[ru];[r]}\hole & (S\otimes S)^2 \ar[d]\ar[r]\ar[rru]|!{[ru];[r]}\hole & (S\otimes S)^2 \ar[d]\ar[r] & (S\otimes S)^2 \ar[d] \\
\cdots \ar[r] & S\otimes S \ar[r]\ar[rru]|!{[ru];[r]}\hole & S\otimes S \ar[r]\ar[rru]|!{[ru];[r]}\hole & S\otimes S \ar[r]\ar[rru]|!{[ru];[r]}\hole & S\otimes S \ar[r] & S\otimes S
}
\]
where the horizontal maps are induce by the differentials of \eqref{eq:complex1}, the vertical maps are the differentials of \eqref{eq:complex2}, the lower slanted maps are given by $\binom{-\Delta(x^m)}{-\Delta(y^n)}$, the upper ones by $(-\Delta(y^n),\Delta(x^m))$.
\end{remark}

Now let us image a noncommutative algebra as the coordinate algebra of a ``noncommutative variety'', and identify the homological (non)smoothness of the former with the (non)smoothness of the latter. Kleinian singularities are nonsmooth in the classical sense; but the algebra $A$ is homologically smooth provided that $\varphi(z)$ has no multiple roots, by Lemma \ref{lem:no-multi-root}. Hence following Hodges's point of view, a homologically smooth $A$ is understood as a \textit{noncommutative resolution of singularity}.

Here we give another sort of deformation such that the (non)smoothness is retained by deformation. We put forward a technical assumption on the field $\kk$ (the idea is motivated by \cite{Van-den-Bergh:Koszul-bimodule-complex}). Let $t$ be equal to $\lambda-1$ if $A$ is a quantum GWA, or $\eta$ if $A$ is a classical GWA.

\textbf{Assumption (A)}: There exists an intermediate field $\mathbb{Q}\subset\kk_0\subset\kk$ such that (i) $\kk_0$ contains $a_l$ and all roots of $\varphi(z)$, (ii) $t$ is a transcendental element over $\kk_0$, (iii) $\kk=\kk_0(\!(t)\!)$.

Under \textbf{Assumption (A)}, consider the $\kk_0[[t]]$-algebra $\tilde{\mathfrak{A}}$ whose generators are $x$, $y$, $z$, subject to $yx=\varphi(z)$ as well as
\begin{align*}
[x,z]&=tzx,\qquad [z,y]=tyz,\\
[x,y]&=\varphi(z+tz)-\varphi(z)=tz\varphi'(z)+O(t^2),
\end{align*}
or
\begin{align*}
[x,z]&=tx,\qquad [z,y]=ty,\\
[x,y]&=\varphi(z+t)-\varphi(z)=t\varphi'(z)+O(t^2),
\end{align*}
depending on whether $A$ is quantum or classical. Since $\varphi(z)\in\kk_0[z]$, by \cite[\S 1]{Gerstenhaber-Schack:deformation} there exists a formal deformation $*$ of
\[
\mathfrak{A}:=\tilde{\mathfrak{A}}/t\tilde{\mathfrak{A}}\cong\kk_0[x,y,z]/(yx-\varphi(z))
\]
such that $(\mathfrak{A}[[t]],*)$ is isomorphic to the $t$-adic completion of $\tilde{\mathfrak{A}}$ and $*$ results in a locally finite deformation which is isomorphic to $A$.

Note that $\mathfrak{A}$ is a commutative GWA. By Jacobian criterion, $\mathfrak{A}$ is smooth if and only if $\varphi(z)$ has no multiple roots, and hence if and only if $A$ is homologically smooth. Under \textbf{Assumption (A)}, the $\kk(\!(\tau)\!)$-algebra $\tilde{A}_\tau$ in subsection \ref{subsec:deformation} can be regarded as ``deformation of deformation'' of $\mathfrak{A}$.

\subsection{Van den Bergh duality}

Let $A$ be homologically smooth, $M$ an $A$-bimodule and let $\Tot\calQ^{\cdot\cdot}$, $\alpha(z)$, $\beta(z)$ be as in Sect.\ \ref{sec:homological-smoothness}. Suppose that $(m_1,m_2,m_3,m_4)\in M^{\oplus 4}$ is any Per $2$-cocycle, then
\begin{align*}
\partial_h^{11}(m_1,m_2)+\partial_v^{2}(m_3,m_4)&=0,\\
s^{1}(m_1,m_2)+\partial_h^{20}(m_3,m_4)&=0.
\end{align*}

It follows that
\begin{align}
-ym_1-\lambda m_2x+zm_3-m_3z&=0,\label{eq:VdB-duality-1}\\
-m_1y-\lambda xm_2+\sigma(z)m_4-m_4\sigma(z)&=0,\label{eq:VdB-duality-2}\\
-{}^{\sigma}\!\Delta(\varphi)\cdot m_1+xm_3-m_4x&=0,\label{eq:VdB-duality-3}\\
-\lambda\Delta^{\sigma}(\varphi)\cdot m_2-m_3y+ym_4&=0.\label{eq:VdB-duality-4}
\end{align}
By induction, we obtain from \eqref{eq:VdB-duality-1}, \eqref{eq:VdB-duality-2} that for any $j\ge 1$,
\begin{align*}
z^jm_3-m_3z^j&=y({}^{\sigma}\!\Delta(z^{j})\cdot m_1)+\lambda(\Delta^{\sigma}(z^{j})\cdot m_2)x,\\
\sigma(z)^jm_4-m_4\sigma(z)^j&=({}^{\sigma}\!\Delta(z^{j})\cdot m_1)y+\lambda x(\Delta^{\sigma}(z^{j})\cdot m_2).
\end{align*}
Thus
\begin{align}
\Delta(\varphi)\cdot (m_3\beta)&=m_3\beta\varphi'+y({}^{\sigma}\!\Delta^D(\varphi)\cdot m_1)\beta+\lambda(\Delta^{\sigma D}(\varphi)\cdot m_2)\sigma(\beta)x,\label{eq:VdB-duality-5}\\
{}^{\sigma}\!\Delta^{\sigma}(\varphi)\cdot (m_4\sigma(\beta))&=m_4\sigma(\beta\varphi')+({}^{\sigma}\!\Delta^D(\varphi)\cdot m_1)\beta y+\lambda x(\Delta^{\sigma D}(\varphi)\cdot m_2)\sigma(\beta).\label{eq:VdB-duality-6}
\end{align}
Let $n_1=-m_3\beta\in M$. Then by \eqref{eq:VdB-duality-5} the first component of $s^{0}(n_1)$ is
\begin{align*}
  &\varphantom{=}m_3-m_3\alpha yx+y({}^{\sigma}\!\Delta^D(\varphi)\cdot m_1)\beta+\lambda(\Delta^{\sigma D}(\varphi)\cdot m_2)\sigma(\beta)x\\
  &=m_3+y({}^{\sigma}\!\Delta^D(\varphi)\cdot m_1)\beta-(m_3\alpha y-\lambda(\Delta^{\sigma D}(\varphi)\cdot m_2)\sigma(\beta))x.
\end{align*}
Denote $n_3=-({}^{\sigma}\!\Delta^D(\varphi)\cdot m_1)\beta$ and $n_4=m_3\alpha y-\lambda(\Delta^{\sigma D}(\varphi)\cdot m_2)\sigma(\beta)$. The first component of $s^{0}(n_1)+\partial_h^{10}(n_3,n_4)$ is $m_3$.

Next we will compute $\partial_h^{01}(n_1)+\partial_v^{1}(n_3,n_4)$. By \eqref{eq:inj.der}, we have
\begin{align*}
\sigma(z)n_3-n_3z&=m_1\beta\varphi'-({}^{\sigma}\!\Delta(\varphi)\cdot m_1)\beta,\\
zn_4-n_4\sigma(z)&=zm_3\alpha y-m_3\alpha y\sigma(z)+\lambda m_2\sigma(\beta\varphi')-\lambda(\Delta^{\sigma}(\varphi)\cdot m_2)\sigma(\beta).
\end{align*}
So by \eqref{eq:VdB-duality-1}--\eqref{eq:VdB-duality-4},
\begin{align*}
  &\varphantom{=}\partial_h^{01}(n_1)+\partial_v^{10}(n_3,n_4)\\
  &=(xm_3\beta -\lambda m_3\beta x+m_1\beta\varphi'-({}^{\sigma}\!\Delta(\varphi)\cdot m_1)\beta,\, ym_3\beta-\lambda^{-1}m_3\beta y\\
  &\varphantom{=}{}+\lambda^{-1}zm_3\alpha y -\lambda^{-1}m_3z\alpha y+m_2\sigma(\beta\varphi')-(\Delta^{\sigma}(\varphi)\cdot m_2)\sigma(\beta))\\
  &=(m_1\beta\varphi'+m_4x\beta-\lambda m_3\beta x, \, ym_3\beta-\lambda^{-1}ym_4\sigma(\beta)+\lambda^{-1}zm_3\alpha y\\
  &\varphantom{=}{}-\lambda^{-1}m_3z\alpha y+m_2-m_2x\alpha y)\\
  &=(m_1\beta\varphi'+m_4x\beta-\lambda m_3\beta x, \, ym_3\beta-\lambda^{-1}ym_4\sigma(\beta)+\lambda^{-1}ym_1\alpha y+m_2)\\
  &=(m_1,\, m_2)-(m_1\alpha yx+\lambda m_3\beta x-m_4\sigma(\beta)x,\, -\lambda^{-1}ym_1\alpha y-ym_3\beta\\
  &\varphantom{=}{}+\lambda^{-1}ym_4\sigma(\beta)).
\end{align*}
Denote $n_2=\lambda^{-1}m_1\alpha y+m_3\beta-\lambda^{-1}m_4\sigma(\beta)$. Then
\[
(m_1,m_2)=\partial_h^{01}(n_1)+\partial_v^{1}(n_3,n_4)+(\lambda n_2x, -yn_2).
\]

Finally, let us compute $(m_3,m_4)-s^{0}(n_1)-\partial_h^{10}(n_3,n_4)$. To the end, it suffices to consider the second component. By \eqref{eq:VdB-duality-3}, \eqref{eq:VdB-duality-6},
\begin{align*}
&\varphantom{=}m_4+\lambda{}^{\sigma}\!\Delta^{\sigma}(\varphi)\cdot n_1-n_3y-xn_4\\
&=m_4-\lambda{}^{\sigma}\!\Delta^{\sigma}(\varphi)\cdot(m_3\beta)+({}^{\sigma}\!\Delta^D(\varphi)\cdot m_1)\beta y-xm_3\alpha y\\
&\varphantom{=}{}+\lambda x(\Delta^{\sigma D}(\varphi)\cdot m_2)\sigma(\beta)\\
&=m_4-{}^{\sigma}\!\Delta^{\sigma}(\varphi)\cdot(\lambda n_2-m_1\alpha y+m_4\sigma(\beta))+({}^{\sigma}\!\Delta^D(\varphi)\cdot m_1)\beta y\\
&\varphantom{=}{}-xm_3\alpha y+\lambda x(\Delta^{\sigma D}(\varphi)\cdot m_2)\sigma(\beta)\\
&=m_4-\lambda{}^{\sigma}\!\Delta^{\sigma}(\varphi)\cdot n_2+({}^{\sigma}\!\Delta(\varphi)\cdot m_1)\alpha y-m_4\sigma(\beta\varphi')-xm_3\alpha y\\
&=m_4\sigma(\alpha)xy-\lambda{}^{\sigma}\!\Delta^{\sigma}(\varphi)\cdot n_2-m_4x\alpha y\\
&=-\lambda{}^{\sigma}\!\Delta^{\sigma}(\varphi)\cdot n_2,
\end{align*}
so $(m_3,m_4)=s^{0}(n_1)+\partial_h^{10}(n_3,n_4)+(0,-\lambda{}^{\sigma}\!\Delta^{\sigma}(\varphi)
\cdot n_2)$. It follows that
\begin{align*}
  (m_1,m_2,m_3,m_4)&=\partial^1(n_1,n_3,n_4)+(\lambda n_2x, -yn_2, 0,-\lambda{}^{\sigma}\!\Delta^{\sigma}(\varphi)\cdot n_2)\\
  &=\partial^1(n_1,n_3,n_4)+f(n_2)
\end{align*}
where $f$ is the map defined at the beginning of this section.

Define $g\colon \Ker\partial^2\to M$ by
\[
  g(m_1,m_2,m_3,m_4)=\lambda^{-1}m_1\alpha y+m_3\beta(z) -\lambda^{-1}m_4\sigma(\beta(z)).
\]

\begin{lem}
  If $A$ is homologically smooth, then we have
  \begin{enumerate}
    \item $g(\Ima\partial^1)\subset [A,M^{\nu}]$,
    \item $fg-\id\subset\Ima\partial^1$,
    \item $gf-\id\subset[A,M^{\nu}]$.
  \end{enumerate}
\end{lem}

\begin{proof}
  (2) follows from the definitions of $f$ and $g$, (3) is easy to check. For (1), the first, third, fourth components of $\partial^1(m_1,m_2,m_3)$ are $-xm_1+\lambda m_1x+\sigma(z)m_2-m_2z$, $-\Delta(\varphi)\cdot m_1+ym_2+m_3x$, $-\lambda{}^{\sigma}\!\Delta^{\sigma}(\varphi)\cdot m_1+m_2y+xm_3$, respectively. Thus
  \begin{align*}
    &\varphantom{=}g(\partial^1(m_1,m_2,m_3))\\
    &=\lambda^{-1}(-xm_1+\lambda m_1x+\sigma(z)m_2-m_2z)\alpha y+(-\Delta(\varphi)\cdot m_1+ym_2\\
    &\varphantom{=}{}+m_3x)\beta-\lambda^{-1}(-\lambda{}^{\sigma}\!\Delta^{\sigma}(\varphi)\cdot m_1+m_2y+xm_3)\sigma(\beta)\\
    &=-\lambda^{-1}xm_1\alpha(z)y+m_1x\alpha y+zm_2\alpha y-m_2\alpha yz-(\Delta(\varphi)\cdot m_1)\beta\\
    &\varphantom{=}{}+ym_2\beta+m_3x\beta+({}^{\sigma}\!\Delta^{\sigma}(\varphi)\cdot m_1)\sigma(\beta)-\lambda^{-1}m_2\beta y-\lambda^{-1}xm_3\sigma(\beta)\\
    &=[x,-\lambda^{-1}m_1\alpha y]-m_1\alpha yx+m_1x\alpha y+[z,m_2\alpha y]-(\Delta(\varphi)\cdot m_1)\beta\\
    &\varphantom{=}{}+[y,m_2\beta]+({}^{\sigma}\!\Delta^{\sigma}(\varphi)\cdot m_1)\sigma(\beta)-[x,\lambda^{-1}m_3\sigma(\beta)]\\
    &=[x,-\lambda^{-1}m_1\alpha y]+[z,m_2\alpha y]-[z,(\Delta^D(\varphi)\cdot m_1)\beta ]+[y,m_2\beta]\\
    &\varphantom{=}{}+[z,\lambda({}^{\sigma}\!\Delta^{\sigma D}(\varphi)\cdot m_1)\sigma(\beta)]-[x,\lambda^{-1}m_3\sigma(\beta)]\\
    &\in [A,M^{\nu}].
  \end{align*}
\end{proof}

\begin{prop}\label{prop:vbd-duality}
  Let $H(f)\colon H_0(A,M^{\nu}) \to H^2(A,M)$ be induced by $f\colon M\to \Ker\partial^2$.
   \begin{enumerate}
     \item $H(f)$ is an isomorphism if $A$ is homologically smooth, whose inverse is induced by $g$.
     \item When $M=A$, $H(f)$ is injective if $A$ is quantum with $\lambda$ not a root of unity.
     \item When $M=A$, $H(f)$ is bijective if $A$ is classical.
   \end{enumerate}
\end{prop}

\begin{proof}
  (1) follows from the previous lemma. When $l\geq 1$, (2), (3) are implicit conclusions of \cite{Solotar:Hochschild-homology-GWA-quantum} and \cite{Farinati:Hochschid-homology-GWA} respectively. When $l=0$, by Theorem \ref{thm:twisted-CY}, $A$ is homologically smooth. Thus (2), (3) hold by (1).
\end{proof}

Recall the Per $2$-cocycles $f(z)$ and $f(1)$ used in subsection \ref{subsec:deformation}. They may give rise to the trivial deformations if they represent zero in $H^2(A,A)$. By Proposition \ref{prop:vbd-duality}, it suffices to consider if $z$ and $1$ are zero in $H_0(A,A^{\nu})=A/[A,A^{\nu}]$.

For any $b\in A$, denote by $\llb b\rrb$ the homology class presented by $b$ in $A/[A,A^{\nu}]$.

Consider the quantum case: $A=\kk[z;\lambda,0,\varphi(z)]$. Let $e$ be the order of $\lambda$ if $\lambda$ is a root of unity, or zero otherwise. Enlarge $\kk$ if necessary, assume that $\kk$ contains all roots $z_1,\ldots,z_l$ of $\varphi(z)$. Denote by $V(z_1,\ldots,z_l)$ the following Vandermonde Matrix indexed by $\mathbb{Z}^+$
\[
\begin{pmatrix}
1 & z_1 & z_1^2 & z_1^3 & \cdots \\
1 & z_2 & z_2^2 & z_2^3 & \cdots \\
\vdots & \vdots & \vdots & \vdots & \ddots \\
1 & z_l & z_l^2 & z_l^3 & \cdots \\
\end{pmatrix}.
\]
Let $V_{\lambda}(z_1,\ldots,z_l)$ be the sub-matrix obtained by picking the $(ke-e+2)$nd column of $V(z_1,\ldots,z_l)$ as the $k$th column for all $k\in\mathbb{Z}^+$. Then we have
\[
V_{\lambda}(z_1,\ldots,z_l)=\mathrm{diag}(z_1,\ldots,z_l)V(z_1^{e},\ldots,z_l^{e}),
\]
whose rank is denote by $R$. By convention, let $R=0$ if $l=0$.

For any $i$, $j>0$, $\llb x^j\rrb=\llb x^{j-1}\nu(x)\rrb=\lambda\llb x^j\rrb$, $\llb z^ix^j\rrb=\llb z^{i-1}x^jz\rrb=\lambda^j\llb z^ix^j\rrb$, and $\llb z^ix^j\rrb=\lambda^{-1}\llb xz^ix^{j-1}\rrb=\lambda^{i-1}\llb z^ix^j\rrb$. So $\llb z^ix^j\rrb\ne 0$ if $i\in e\mathbb{N}+1$, $j\in e\mathbb{N}$. An analogous discussion holds for $\llb z^iy^j\rrb$.

It follows from $\llb yxz^n\rrb=\lambda^{-1}\llb xz^ny\rrb=\lambda^{n-1}\llb xyz^n\rrb$ that
\begin{equation*}
\sum_{i=0}^{l}a_i(1-\lambda^{i+n-1})\llb z^{i+n}\rrb=0.
\end{equation*}
Let $S_n=(1-\lambda^{n-1})\llb z^{n}\rrb$. Then $\{S_n\}_{n\in\mathbb{N}}$ is a linear recursive sequence, and so we have
\begin{equation}\label{eq:general-term-formula}
(S_0,S_1,\ldots)=(T_0,\ldots,T_{l-1})V(z_1,\ldots,z_l)
\end{equation}
for some $T_i\in A/[A,A^{\nu}]$. Since $S_{ke+1}=0$ for all $k\ge 0$, it follows from \eqref{eq:general-term-formula} that
\[
(T_0,T_1,\ldots,T_{l-1})V_{\lambda}(z_1,\ldots,z_l)=0.
\]

Observe that $R=\#\{t\in\kk^\times \mid t=z_i^e \text{ for some $1\leq i\leq l$}\}$. By rearranging the order of the roots if necessary, we may assume that $z_1=\cdots=z_p=0$ if $\varphi(z)$ has $0$ as a root of multiplicity $p$, and that $z_{l-R+1}^e,\ldots,z_{l-1}^e,z_l^e$ are distinct nonzero numbers. In this way, $\{T_0,\ldots,T_{l-R-1}\}$ is a maximal linearly independent subset of $\{T_0,\ldots,T_{l-1}\}$.

Define $\xi\colon\mathbb{Z}^+\to\mathbb{N}$ by $\xi(1)=0$ and
\begin{alignat*}{2}
&\text{$\xi(k(e-1)+c)=ke+c$ for all $k\geq 0$ and $2\leq c\leq e$},&\quad &\text{if $e\geq 2$},\\
&\text{$\xi(n)=n$ for all $n\ge 2$}, &\quad &\text{if $e=0$}.
\end{alignat*}
Then by \eqref{eq:general-term-formula}, $\{S_{\xi(1)},\ldots,S_{\xi(l-R)}\}$ is a maximal linearly independent subset of $\{ S_0,S_1,\ldots\}$. Consequently,
\[
A/[A,A^{\nu}]=\bigoplus_{i=1}^{l-R}\kk\llb z^{\xi(i)}\rrb\oplus\bigoplus_{j\in e\mathbb{N}}\bigoplus_{k\in e\mathbb{Z}}\kk\llb z^{j+1}x_{k}\rrb,
\]
and so $\llb z\rrb\neq 0$, independently of $\lambda$ and $\varphi(z)$.

Next we consider the classical case: $A=\kk[z;1,\eta,\varphi(z)]$. Since $\nu=\id$, the quotient $A/[A,A^\nu]$ is $H_0(A,A)$. By \cite{Farinati:Hochschid-homology-GWA},
\begin{equation}\label{eq:basis-classical}
H_0(A,A)=
\begin{cases}
\vphantom{\bigg[]}0, & \text{ if $l=0$, $1$,}\\
\displaystyle\bigoplus_{i=0}^{l-2}\kk\llb z^i\rrb, & \text{ if $l\ge 2$.}
\end{cases}
\end{equation}
So $\llb 1\rrb\neq 0$ if and only if $\deg \varphi(z)\geq 2$.

\begin{remark}
Although \cite{Farinati:Hochschid-homology-GWA} excludes the case $l=0$, the fact $H_0(A,A)=0$ stays valid in this case.
\end{remark}

Let us close the section by a proposition which claims that the formal deformations $(A[[\tau]],*)$ constructed in subsection \ref{subsec:deformation} are not equivalent to the trivial one in most cases.

\begin{prop}
Let $A_1$, $A_2$ be the algebras $(A[[\tau]],*)$ given by Theorems \ref{thm:formal-deformation-quantum} and \ref{thm:formal-deformation-classical} respectively. Let $A_3=(A[[\tau]],*_{\mathrm{tr}})$. We have $A_1\ncong_\mathrm{f} A_3$ for any $\lambda$, $\varphi(z)$, and $A_2\ncong_\mathrm{f} A_3$ if and only if $\deg\varphi(z)\geq 2$.
\end{prop}

\begin{proof}
When $A$ is classical and $\deg\varphi(z)< 2$, we have $H^2(A,A)=0$. By \cite[Theorem, \S 5]{Gerstenhaber-Schack:deformation}, every formal deformation of $A$ is equivalent to the trivial one, i.e., $A_2\cong_\mathrm{f} A_3$.

By the computation of $H_0(A,A^\nu)$ and Proposition \ref{prop:vbd-duality}, we have $A_1\ncong_\mathrm{f} A_3$ if $\lambda$ is not a root of unity, and $A_2\ncong_\mathrm{f} A_3$ if $\deg\varphi(z)\geq 2$.

Suppose that $\lambda$ is a primitive $e$th root of unity. If $A_1$ were isomorphic to $A_3$, their first order deformations were also isomorphic, i.e., $A_1/\tau^2 A_1\cong_\mathrm{f} A_3/\tau^2 A_3$. The former is $\kk[\tau]\langle x,y,z\rangle/(f_0,f_1,f_2,f_3,f_4)$ where
\begin{align*}
f_0&=\tau^2,\\
f_1&=xz-(1-\tau)\lambda zx,\\
f_2&=yz-(1+\tau)\lambda^{-1}zy,\\
f_3&=xy-\bar{\varphi}(z)+\tau z\bar{\varphi}'(z),\\
f_4&=yx-\varphi(z).
\end{align*}
Since $z^e$ is a central element of $A$, it is also central in $A_3/\tau^2 A_3$. It follows that there exists $b\in A$ such that $z^e+b\tau$ is central in $A_1/\tau^2 A_1$. However, a direct computation shows that this is impossible. Hence $A_1$ cannot be isomorphic to $A_3$.
\end{proof}

\section*{Acknowledgments}
I am grateful to Prof.\ Wendy Lowen for her inspiring discussions and helpful conversations. Some ideas arose from my doctoral research. I would like to thank my supervisor Prof.\ Quanshui Wu for his advice and encouragement.

%\bibliographystyle{amsplain}
%\bibliography{GWA}

\end{document}